\documentclass[12pt]{article}
\usepackage{amsmath,amsfonts,amssymb}
\usepackage{alltt}

\usepackage{amsmath,amsfonts,ifthen,fullpage,enumerate}  %,amsthm}

\usepackage{amsmath,amsfonts,amssymb}

\usepackage{makecell}  %thick hline
\usepackage{graphicx}        % standard LaTeX graphics tool
\usepackage[table]{xcolor}   % For coloured cells.

\def\spam{\mathop{\rm span}\nolimits}

\def\Harm{\mathop{\rm Harm}\nolimits}
\def\Ang{\mathop{\rm Ang}\nolimits}
\def\core{\mathop{\rm harm}\nolimits}

\def\trace{\mathop{\rm trace}\nolimits}

\def\geg{{\mathop{\rm geg}\nolimits}}
\def\cgeg{{\DD,d}}

\def\rank{\mathop{\rm rank}\nolimits}
\def\TP{\mathop{\rm TP}\nolimits}
\def\diag{\mathop{\rm diag}\nolimits}
\def\supp{\mathop{\rm supp}\nolimits}

\def\pmat#1{\begin{pmatrix}#1\end{pmatrix}}

\def\question#1{{\bf Question: }#1}
\def\question#1{}

\def\rev{{\rm rev}}

\def\cC{{\cal C}}
\def\cE{{\cal E}}
\def\cS{{\cal S}}
\def\cF{{\cal F}}

\def\R{\mathbb{R}}

\def\CC{\mathbb{C}}
\def\NN{\mathbb{N}}

\def\DD{\mathbb{D}}
\def\FF{\mathbb{F}}

\def\Cd{\C^d}
\def\Rd{\R^d}

\def\C{\mathbb{C}}

\def\SS{\mathbb{S}}

\newcommand{\RR}{\mathbb{R}}

\newtheorem{theorem}{Theorem}[section]
\newtheorem{corollary}{Corollary}[section]
\newtheorem{lemma}{Lemma}[section]
\newtheorem{example}{Example}[section]
\newtheorem{remark}{Remark}[section]
\newtheorem{proposition}{Proposition}[section]

\newenvironment{proof}{{\noindent \it
Proof.}}{\hfill$\Box$\medskip}
\newif\ifdraft\def\draft{\drafttrue\hoffset=.8truecm\showlabeltrue
\def\comment##1{{\bf comment: ##1}}
\headline={\sevenrm \hfill \ifx\filenamed\undefined\jobname\else\filenamed\fi%
(.tex) (as of \ifx\updated\undefined???\else\updated\fi)
 \TeX'ed at {\hour\time\divide\hour by 60{}%
\minutes\hour\multiply\minutes by 60{}%
\advance\time by -\minutes
\the\hour:\ifnum\time<10{}0\fi\the\time\  on \today\hfill}}
}

\def\inpro#1{\langle#1\rangle}
\def\ip#1{\langle\kern-.28em\langle#1\rangle\kern-.28em\rangle_\nu}

\def\cP{{\cal P}}
\def\cU{{\cal U}}

\def\cJ{{\cal J}}
\def\cV{{\cal V}}

\def\norm#1{\Vert#1\Vert}%\def\norm#1{\|#1\|} does not work with verbatim.tex
\def\openR{{{\rm I}\kern-.16em {\rm R}}}

\def\Fd{\FF^d}
\let\ga\alpha

\let\gb\beta

\let\gG\Gamma
\let\gd\delta
\let\gD\Delta
\let\gep\varepsilon

\let\gth\theta

\let\gl\lambda

\let\gL\Lambda

\let\gs\sigma

\let\go\omega
\let\gO\Omega
\let\ga\alpha

\let\gG\Gamma
\let\gb\beta
\let\gd\delta
\let\gs\sigma
\def\inpro#1{\langle#1\rangle}

\def\Hom{\mathop{\rm Hom}\nolimits}

\def\Iff{\hskip1em\Longleftrightarrow\hskip1em}
\def\Implies{\hskip1em\Longrightarrow\hskip1em}

\def\formeq{\the\sectionno.\the\equationno}  %% equation numbering
\def\elabel#1/#2/#3/{\global\advance\equationno by 1 %
\ifx#1\empty\else\emember#1%
\ifshowlabel\marginal{\string#1}\fi\fi%
\ifmmode\eqno{#3(\formeq#2)}\else#3\formeq#2\fi} %<------ switch to \leqno ???

\def\makeblanksquare#1#2{
\dimen0=#1pt\advance\dimen0 by -#2pt
      \vrule height#1pt width#2pt depth0pt\kern-#2pt
      \vrule height#1pt width#1pt depth-\dimen0 \kern-#1pt
      \vrule height#2pt width#1pt depth0pt \kern-#2pt
      \vrule height#1pt width#2pt depth0pt
}
\title{\bf 
Real and complex spherical designs and their Gramian
}

\author{Shayne Waldron\\
 \\
Department of Mathematics \\ University of Auckland\\
Private
Bag 92019, Auckland, New Zealand\\
e-mail: waldron@math.auckland.ac.nz}
\date{\today}

\begin{document}

\maketitle 

\begin{abstract}
If a (weighted) spherical design is defined as an integration 
(cubature) rule for a unitarily invariant space $P$ of polynomials 
(on the sphere),
then any unitary image of it is also such a spherical design. 
It therefore follows that such spherical designs are determined by their
Gramian (Gram matrix). We outline a general method to obtain such
a characterisation as the minima of a function of the Gramian,
which we call a potential. 
%For various commonly considered 
%spherical designs, we give a corresponding characterisation in terms
%of the Gramian. 
This characterisation
%ese equations are as simple as possible, and 
can be used for the numerical and analytic construction of spherical designs. 
When the space $P$ of polynomials 
is not irreducible under the action of the unitary group,
then the potential is not unique. 
In several cases of interest,
e.g., spherical $t$-designs and half-designs, we use this flexibility
to provide potentials with a very simple form.
We then use our results to develop certain aspects of the theory 
of real and complex spherical designs for unitarily invariant 
polynomial spaces. 
%	This includes giving the complex analogue
%of many classical results, 
%such as the upper/absolute/special and lower/Fisher type bounds on the
%number of points in a design,
%and showing that the projective designs are a special case of our general results.

%We discuss the close relationship between our approach
%and the classical results estimating the number of vectors in 
%spherical codes and designs.
%Using this as a model, we then extend these ideas to other designs 
%such as combinatorial designs and complex spherical $\tau$-designs,
%giving an alternative approach to that of Delsarte spaces.

%This is the first explicitly known spherical half-design of order $8$ for a dimension that is greater than $2$.
\end{abstract}

\bigskip
\vfill


%pentakis dodecahedron

 \ifdraft Combinatorial aspects of finite geometries \else\fi
05B30, \ifdraft (Other designs, configurations) \else\fi
42C15, \ifdraft General harmonic expansions, frames  \else\fi
65D30; \ifdraft (Numerical integration) \else\fi
\quad
secondary
94A12. \ifdraft (Signal theory [characterization, reconstruction, etc.]) \else\fi
51M20, \ifdraft Polyhedra and polytopes; regular figures, division of spaces \else\fi

\vskip .5 truecm
\hrule
\newpage

\section{Introduction}

%Let $\SS$ be the unit sphere in $\Rd$ or $\Cd$
%and $\gs$ be the normalised surface area measure on $\SS$.
Let $\gs$ be the normalised surface area measure on the unit sphere $\SS$
in $\Rd$ or $\Cd$.
A ({\bf weighted}) {\bf spherical design} (for $P$) is a sequence of points $v_1,\ldots,v_n$ in $\SS$
and {\bf weights} $w_1,\ldots,w_n\in\RR$, %$w_1,\ldots,w_n\ge0$, 
$w_1+\cdots+w_n=1$,
for which the integration (cubature) rule 
\begin{equation}
\label{cuberule}
\int_\SS p(x)\, d\gs(x) = \sum_{j=1}^n  w_j p(v_j),
\end{equation}
holds for all $p$ in a %some 
finite dimensional space of polynomials $P$.
This is essentially the first of the four definitions given in \cite{S01}.
We say that a polynomial $p$ (or a space of polynomials) is
{\bf integrated} by the spherical design (integration/cubature rule)
%$(v_j)$ (of order $m$) 
if (\ref{cuberule}) holds.
The existence of %such a spherical design 
spherical designs for constant weights, i.e.,
 $w_j={1\over n}$,
and  $n$ sufficiently large, 
was proved in  \cite{SZ84}.
In applications, it is often required that $w_j\ge0$.  % weights be nonnegative.
The common choices for $P$ are unitarily invariant, 
i.e.,  for $U$ unitary, $p\circ U\in P$, $\forall p\in P$.
In the real case, the unitary (inner product preserving) maps are the orthogonal transformations. 
For such a space $P$, the unitary invariance of the measure $\gs$ implies
that $(Uv_j),(w_j)$ is also a spherical design when $U$ is  unitary,  
by the calculation
$$ \sum_{j} w_j p(Uv_j)  %= \sum_{j} w_j (p\circ U) (v_j)
=\int_\SS (p\circ U)(x)\, d\gs(x)
%=\int_\SS p(Ux)\, d\gs(x)
=\int_\SS p(x)\, d\gs(x), \qquad \forall p\in P. $$

The {\bf Gramian} (or {\bf Gram matrix}) of a sequence of vectors
$(v_j)$ is the Hermitian 
matrix of inner products $[\inpro{v_j,v_k}]$.
We say that sequences of vectors $(v_j)$ and $(u_j)$ are 
{\bf unitarily equivalent} if $u_j=Uv_j$, $\forall j$, where $U$
is unitary. 
Since
$$ \inpro{u_j,u_k}=\inpro{Uv_j,Uv_k}=\inpro{v_j,v_k}, $$
a sequence of vectors is defined up to unitary equivalence by its 
Gramian (see \cite{W18}). 
Combining this with our previous observation gives:

\begin{itemize}
\item The real or complex spherical designs $(v_j)$ for a unitarily invariant polynomial space
$P$ are determined by the Gramian of $(v_j)$.
\end{itemize}

It should therefore be possible to express the condition of $(v_j)$ being a spherical 
design in terms of $\inpro{v_j,v_k}$, $1\le j,k\le n$, and the weights $(w_j)$ if these are 
not constant. 
The primary objective of this paper is to give such a characterisation 
for being a spherical design, which is as simple as possible.
%, and to extend it to more general designs. 
Some key features of our approach are:
\begin{itemize}
\item Because the reproducing kernel $K(x,y)$ for a unitarily invariant 
polynomial space
%of polynomials 
$P$ is a function of $\inpro{x,y}$, we are able to find a 
``potential'' $A_P\bigl([\inpro{v_j,v_k}])\ge0$ 
whose zeros are the spherical designs for $P$.
\item The potential $A_P$ can be given by a univariate polynomial $F$,
which is not unique when $P$ is not irreducible under the action
of the unitary group (see Table \ref{P-designlist}).
%\item For $P$ the space of homogeneous polynomials of degree $m$, % on $\Rd$, 
\item The polynomial space $P$ and the corresponding real or complex spherical designs
can be described by a finite set of indices $L\subset\NN$ or $\tau\subset\NN^2$,
respectively, which index the irreducible subspaces of $P$.
\end{itemize}

Motivated by a careful analysis of the real 
spherical designs \cite{BB09},
% (adopting some of the notation of \cite{DGS77}),
we obtain a unified theory of the most general real and complex
spherical designs, which includes:
%We motivate our development by a careful analysis of the real 
%spherical designs \cite{BB09}
% (adopting some of the notation of \cite{DGS77}).  
%The result is a unified theory of the most general real and complex
%spherical designs. Some highlights include:

\begin{itemize}
%\item A general variational characterisation that applies to all types of designs.
\item A general variational characterisation for all types of designs (Theorem \ref{varcharthm}).
%\item New characterisations for real spherical designs, including $t$-designs and half-designs.
\item New characterisations for real spherical designs, including $t$-designs
	(Theorem \ref{tdesignvarthm}) and half-designs (Theorem \ref{mweightsuffthm}).
\item New characterisations for complex spherical designs (Theorem \ref{HpqF},
Theorem \ref{weightedversionHompq}).
\item A natural description for real and complex projective spherical designs
(\S \ref{projectivedesignsect}).
\item A description of the Gegenbauer polynomials that appear naturally
in the analysis of complex
spherical designs as orthogonal polynomials (of two variables),
together with results about their products and sums 
(\S\ref{Gegenorthsect}, \S\ref{Gegenprods}).
%\item Upper/absolute/special and lower/Fisher type bounds for 
%the number of vectors in a general real or complex spherical design.
\item A unified approach to bounds on the numbers of points in a spherical design
	(\S\ref{boundsection}).
\item Unified results about the
	number of vectors in $A$-sets and $s$-angular sets (\S\ref{morecomplexbounds}).
%\item For $P$ a unitarily invariant space of homogeneous polynomials, 
%%of degree $m$, % on $\Rd$, 
%we single out a particularly nice potential, and give a simple proof
%using tensor products. % (which does not extend to the general case).
%\item For convenience and simple formulas, we allow the vectors $(v_j)$
%to have (possibly non-unit) norms which determine the weights $(w_j)$.
%\item We discuss how the integration of certain subspaces of polynomials
%implies that other (both lower and higher degree) polynomials are also
%integrated.
%\item We discuss how certain cubature rules (for polynomials which 
%integrate to zero) are degenerate, and some others are natural 
%spherical designs to consider.
\end{itemize}
As much as possible, we %try to 
treat the real and complex cases simultaneously,
with $\FF=\RR,\CC$.

\section{Basic definitions}

Sometimes  %e.g., for finding numerical constructions,
it is convenient to describe a weighted spherical design for $P$ as
a sequence of (possibly not unit) vectors $(v_j)$, 
where the weights are inferred from the $\norm{v_j}$ by
\begin{equation}
\label{wjdef}
w_j =w_j^{(m)} := {\norm{v_j}^m\over\sum_\ell \norm{v_\ell}^m},
\quad\hbox{which we call the {\bf $m$-weights}}.
\end{equation}
Designs with constant weights are called {\bf unweighted}, {\bf classical}
or simply {\bf designs}.
%If $(\hat v_j)$ is a ``weighted spherical design'' with weights $(w_j)$,
%and $p$ is a homogeneous polynomial of degree $m$, then
%(\ref{cuberule}) can be written
%$$ \int_\SS p(x)\, d\gs(x) 
%= \sum_{j=1}^n  w_j p(\hat v_j)
%= \sum_{j=1}^n  p(w_j^{1\over m}\hat v_j)
%= \sum_{j=1}^n  p(v_j)
%= \sum_{j=1}^n  \norm{v_j}^m p(\hat v_j),
%\qquad v_j:=w_j^{1\over m}\hat v_j. $$
%In this way, we will describe a weighted spherical design in terms
%of vectors $(v_j)$, with the weights $w_j=\norm{v_j}^m$ inferred
%from $\norm{v_j}$ and $m$ (which we call the order).
%
%\begin{definition} Let $P$ be a unitarily invariant space of polynomials
%on $\Fd$, 
%where $\FF$ is $\RR$ or $\CC$, 
%and $(v_j)$ be a sequence of vectors in $\Fd$.
%%and $m$ be a positive integer. 
%Then we say that $(v_j)$ is a ({\bf weighted})
%{\bf spherical design} %(of {\bf order} $m$) 
%for $P$ if
%\begin{equation}
%\label{wspheredesdef}
%\int_\SS p(x)\,d\gs(x) 
%= \sum_{j=1}^n {\norm{v_j}^2 \over \sum_\ell\norm{v_\ell}^2}\,
%p\Bigl({v_j\over\norm{v_j}}\Bigr), \qquad \forall p\in P.
%\end{equation}
%%where $C:=\sum_\ell\norm{v_\ell}^m$. It is {\bf normalised} if $C=1$.
%\end{definition}
%
%We will primarily be interested in the (unitarily invariant) polynomial spaces
%Particular cases of interest include $P$ the
%(unitarily invariant) polynomial space

It is a subtle but important point, which follows from (\ref{cuberule}), 
that a spherical design depends only the restriction of $P$ to the sphere,
i.e. the space of harmonic polynomials
$$ \core(P):=P|_\SS, $$
which we will call the {\bf harmonic part} of the polynomial space $P$.

Choices for the (unitarily invariant) polynomial space $P$ of interest include:
\begin{align*}
& \Pi_t(\Rd) = \hbox{polynomials of degree $\le t$ on $\Rd$}
\quad \hbox{({\bf spherical $t$-designs})}, \cr
& \Hom_m(\Rd) = \hbox{homogeneous polynomials of degree $m$ on $\Rd$} 
\quad \hbox{({\bf spherical half-designs})}, \cr %of order $m$
& \Harm_m(\Rd) = \hbox{harmonic polynomials in $\Hom_m(\Rd)$} 
\quad \hbox{({\bf spherical designs of harmonic index $m$})}, \cr
%& \Pi_m(\Cd) = \hbox{holomorphic polynomials of degree $\le m$ on $\Cd$} \quad
%\hbox{(complex spherical $m$-designs)}, \cr
& \Hom_{t,t}(\Cd) =\spam\{|\inpro{\cdot,v}|^{2t}:v\in\Cd\} \quad 
\hbox{({\bf spherical $(t,t)$-designs}, {\bf projective $t$-designs})}.  % of order $m=2t$
\end{align*}
The half-designs for 
%$m=2t$, i.e., 
$P=\Hom_{2t}(\Rd)$ are also called
({\bf real}) {\bf spherical $(t,t)$-designs}. 

\begin{example}
%We say that a polynomial $p$ (or a space of polynomials) is 
%{\bf integrated} by the spherical design (integration/cubature rule)
%$(v_j)$ (of order $m$) if (\ref{wspheredesdef}) holds.
From the observation
$$ \int_\SS \norm{x}^{2k}p(x)\, d\gs(x)
= \int_\SS p(x)\, d\gs(x)
= \sum_j w_j p\Bigl({v_j\over\norm{v_j}}\Bigr)
= \sum_j w_j \bigl(\norm{\cdot}^{2k} p\bigr)\Bigl({v_j\over\norm{v_j}}\Bigr), $$
it follows that a spherical design for $P$ also integrates the spaces
$$ P^- = \bigl\{ q: \norm{\cdot}^{2j}q\in P, \exists j\ge0 \bigr\}, \qquad
P^+  = \bigl\{ \norm{\cdot}^{2k}p: p\in P,k\ge0\bigr\}, $$
with $\core(P^-)=\core(P^+)=\core(P)$. 
%$$ \core(P^-)=\core(P^+)=\core(P). $$
In particular, for $P=\Hom_m(\Rd)$, we have
%$$ P^-+P^+ = \cdots \Pi_{m-2}^\circ(\Rd)\oplus\Pi_{m}^\circ(\Rd)\oplus
%\Pi_{m+2}^\circ(\Rd)\oplus\Pi_{m+4}^\circ(\Rd)\oplus\cdots. $$
%\begin{align*}
%P^- &= \Pi_m^\circ(\Rd)\oplus\Pi_{m-2}^\circ(\Rd)\oplus \Pi_{m-4}^\circ(\Rd)\oplus\cdots \cr
%P^+ &= \Pi_m^\circ(\Rd)\oplus\Pi_{m+2}^\circ(\Rd)\oplus \Pi_{m+4}^\circ(\Rd)\oplus\cdots.
%\end{align*}
$$ P^- = \Hom_m(\Rd) \oplus \Hom_{m-2}(\Rd) \oplus \Hom_{m-4}(\Rd) \oplus \cdots , \quad
P^+ = \bigoplus_{k=0}^\infty \norm{\cdot}^{2k}\Hom_m(\Rd).  $$
%For this reason, \cite{KP11} call the spherical designs for $P=\Hom_m(\Rd)$
Because of this, \cite{KP11} refer to the spherical designs for $P=\Hom_m(\Rd)$
as the {spherical half-designs} (of order $m$).
\end{example}

Every homogeneous polynomial $p\in\Hom_k(\Rd)$ of degree $k$ 
can be written uniquely
\begin{equation}
\label{homharmrep}
p(x) = \sum_{0\le j\le{k\over2}} \norm{x}^{2j} p_{k-2j}(x)
= \sum_{j=0}^{[{k\over2}]} \norm{x}^{2j} p_{k-2j}(x), 
%\quad x\in\Rd,
\end{equation}
where $p_{k-2j}\in\Harm_{k-2j}(\Rd)$, % (see \cite{SW71} Th.\ 2.1).
and the restriction map
$$\Harm_k(\Rd)\to L_2(\SS):f\mapsto f|_\SS$$
is injective, with image denoted $\Harm_k(\SS)$. We will freely identify spaces of 
harmonic functions on $\Rd$ and $\SS$ ({\bf solid} and {\bf surface} spherical harmonics).
It follows from the irreducibility of the summands above \cite{ABR01},
that every unitarily invariant (invariant under orthogonal transformations) polynomial space on $\Rd$
can be written uniquely as a direct sum
$$ P= \bigoplus_{(j,k)\in\cJ} \norm{\cdot}^{2j}\Harm_{k-2j}(\Rd), $$
for a subset $\cJ$ of the indices $\{(j,k):0\le j\le{k\over2}\}$.
For the purpose of integration on % the sphere 
$\SS$, it suffices to 
consider the (possibly lower dimensional) space of harmonic polynomials
\begin{equation}
\label{PLobservation}
\core(P)=P|_\SS = \bigoplus_{\ell\in L} \Harm_{\ell}(\SS)
\approx \bigoplus_{\ell\in L} \Harm_{\ell}(\Rd),
\quad L:=\{k-2j:(j,k)\in\cJ\}\subset\NN.
\end{equation}
%$$  \harm(P) := \bigoplus_{(j,k)\in\cJ} \Harm_{k-2j}(\Rd)
%\approx \bigoplus_{(j,k)\in\cJ} \Harm_{k-2j}(\SS), $$
%which we call the {\bf harmonic part} of $P$. We note that
%$$ \core(P^-)=\core(P^+)=\core(P), $$
We note, in particular
$$ \core(\Pi_m(\Rd)) = \bigoplus_{j=0}^m\Harm_{j}(\Rd), \qquad
\core(\Hom_m(\Rd)) = \bigoplus_{0\le j\le {m\over2}}\Harm_{m-2j}(\Rd). $$
%$\{\norm{\cdot}^{2j}q:q\in\core(P),j\in\{0,1,\ldots\}\}$.
A necessary condition for $(v_j)$ to integrate $\Hom_m(\Rd)$
with the $m$-weights (\ref{wjdef}) is that
\begin{align}
\label{mweights1}
\int_\SS\int_\SS \inpro{x,y}^m d\gs(x)\,d\gs(y)
&= \int_\SS \sum_{j=1}^n {\norm{v_j}^m\over\sum_\ell\norm{v_\ell}^m} \inpro{{v_j\over\norm{v_j}},y}^m d\gs(y) \cr
&= \sum_{j=1}^n \sum_{k=1}^n 
{\norm{v_j}^m\norm{v_k}^m \over(\sum_\ell\norm{v_\ell}^m)^2}
\inpro{{v_j\over\norm{v_j}},{v_k\over\norm{v_k}}}^m \cr
&= {1\over (\sum_\ell\norm{v_\ell}^m)^2}\sum_{j=1}^n \sum_{k=1}^n 
\inpro{v_j,v_k}^m.
\end{align}
%For $m$ even this is known to be sufficient \cite{Si74}.
To also integrate $\Hom_{m-1}(\Rd)$ with these $m$-weights 
a necessary condition is
\begin{align}
\label{mweights2}
\int_\SS\int_\SS \inpro{x,y}^{m-1} d\gs(x)\,d\gs(y)
&= \int_\SS \sum_{j=1}^n {\norm{v_j}^m\over\sum_\ell\norm{v_\ell}^m} \inpro{{v_j\over\norm{v_j}},y}^{m-1} d\gs(y) \cr
&= \sum_{j=1}^n \sum_{k=1}^n 
{\norm{v_j}^m\norm{v_k}^m \over(\sum_\ell\norm{v_\ell}^m)^2}
\inpro{{v_j\over\norm{v_j}},{v_k\over\norm{v_k}}}^{m-1} \cr
&= {1\over (\sum_\ell\norm{v_\ell}^m)^2}\sum_{j=1}^n \sum_{k=1}^n 
\norm{v_j}\norm{v_k} \inpro{v_j,v_k}^{m-1}.
\end{align}
We will show that these conditions (for a weighted spherical $m$-design)
are also sufficient (Theorem \ref{mweightsuffthm}). 
It is most natural to view this result as a special case of a
very general variational characterisation of spherical designs, 
which we now describe.

%By way of motivation, we recall that the unitarily invariant 
%finite dimensional spaces
%of polynomials on the real sphere are
%$$ P_L= \bigoplus_{\ell\in L} \Harm_\ell(\SS), $$
%where $L$ is a finite set of indices.

\section{The variational characterisation of designs}

We generalise the surface area measure $\gs$ on $\SS$ to a measure
$\mu$ on a set $\gO\subset\Fd$. We say that a sequence of points $(v_1,\ldots,v_n)$ 
in $\gO$ and weights
$w=(w_j)$, $w_1+\cdots+w_n=1$, $w_j\ge0$,
is a {\bf weighted} ({\bf spherical}) {\bf design} 
 for a space $P$ of functions $\gO\to\FF$
(or simply a {\bf $P$-design}) if 
$$ \int_\gO p(x)\,d\mu(x) = \sum_{j=1}^n w_j p(v_j), \qquad\forall p\in P. $$
If point evaluation on $P$ is a continuous linear functional with respect
to the inner product
$$ \inpro{f,g}_\mu := \int_\gO f\overline{g}\,d\mu $$
given by $\mu$ 
(as is the case for $P$ finite dimensional),
then it can be represented by the {\bf reproducing kernel} 
$K=K_P:\gO\times\gO\to\FF$,
which is given by 
$$ f(x) = \int_\gO K(x,y)f(y)\,d\mu(y), \qquad\forall f\in P, $$
where $K(x,y) = \sum_s Y_s(x) \overline{Y_s(y)}$
for $(Y_s)$ an orthonormal basis of $P$.
%where if $(Y_s)$ is an orthonormal basis of $P$ then
%$$ K(x,y) = \sum_s Y_s(x) Y_s(y). $$

\begin{theorem}
\label{varcharthm} (Variational characterisation) 
Let $\mu$ be a measure on $\gO$ with $\mu(\gO)=1$, $P$ be a finite dimensional
 space of 
integrable functions on $\gO$, and $H$
be the subspace of functions which are
orthogonal to the constants, i.e.,
$$ H=P\ominus\spam\{1\}=\{p\in P: \int_\gO p \,d\mu =0\}. $$
Let $\Phi=(v_1,\ldots,v_n)$, $v_j\in\gO$, and $w=(w_j)\in\RR^n$ be weights
with $w_1+\cdots+w_n=1$.   %, $w_j\ge0$. 
Write $H$ as a %an orthogonal
direct sum $H=\oplus_\ell H^{(\ell)}$, with
$K_\ell$ the reproducing kernel of $H^{(\ell)}$ and $c_\ell>0$. 
Then
\begin{equation}
\label{Apotent}
%A\bigl((v_j),(w_j))
A_{w,c}(\Phi) = 
A_{P,\mu,w,c}(\Phi) := 
\sum_{j=1}^n\sum_{k=1}^n w_j w_k 
\sum_\ell c_\ell 
K_\ell(v_j,v_k) \ge 0, 
\end{equation}
with equality if and only if $(v_j),(w_j)$ is a weighted spherical
design for $P$.
\end{theorem}

\begin{proof}
We first recall that the reproducing kernel $K_\ell$ for $H^{(\ell)}$
represents the point evaluations, i.e.,
$$ f(x) = \int_\gO K_\ell(x,y)f(y)\,d\mu(y), \qquad\forall f\in H^{(\ell)}, $$
and that
	$$ K_\ell(x,y) = \sum_s Y_s^{(\ell)}(x) \overline{Y_s^{(\ell)}(y)}, $$
where $(Y_s^{(\ell)})$ is an orthonormal basis for $H^{(\ell)}$.
We therefore compute
\begin{align*}
A_{w,c}(\Phi) 
& := 
\sum_{j=1}^n\sum_{k=1}^n w_j w_k
\sum_\ell c_\ell 
 K_\ell(v_j,v_k) \cr
& = \sum_\ell c_\ell \sum_{j=1}^n\sum_{k=1}^n w_j w_k 
\sum_s Y_s^{(\ell)}(v_j) \overline{Y_s^{(\ell)}(v_k)} \cr
&= \sum_{\ell} c_\ell \sum_s
\left(\sum_{j=1}^n w_j Y_s^{(\ell)}(v_j) \right)
\left(\sum_{k=1}^n w_k \overline{Y_s^{(\ell)}(v_k)} \right) \cr
&= \sum_{\ell=1} c_\ell \sum_s
\Bigl\vert \sum_{j=1}^n w_j Y_s^{(\ell)}(v_j) \Bigr\vert^2 \ge 0,
\end{align*}
with equality if and only if
$$ \sum_{j=1}^n w_j Y_s^{(\ell)}(v_j) = 0 = \int_\gO Y_s^{(\ell)}(x)\,d\mu(x), 
\qquad\forall s,\ \forall\ell. $$
which is equivalent 
to $(v_j),(w_j)$ being a $P$-design
(by linearity and the fact $\sum_j w_j=1$ ensures 
that the constants are integrated).
\end{proof}

This result is essentially an abstract version of \cite{SW09} (Theorem 3),
which was for spherical designs on the real sphere, with $P=\Pi_t(\Rd)$.
For examples of such potentials, we will use a variety of subscripts for $A$ indicating
parameters that it depends on, and which we choose to emphasize in a particular context.

We call the $A_{w,c}(\Phi)$ of (\ref{varcharthm})
a {\bf potential} for the $P$-designs (with weights $w$), given by
$$ K:=\sum_\ell c_\ell K_\ell. $$
The potential for $P$ when there is a single summand and $c_1=1$ is 
called the {\bf canonical potential}, and is denoted by $A_P$.
This can be obtained by taking $\oplus_\ell H^{(\ell)}$ to be an orthogonal 
direct sum and $c_\ell=1$, $\forall\ell$.
It is also convenient at times
to add a (positive) constant $b_0$ to a potential, to obtain
a {\bf potential with constant} $B(\Phi)=A_{w,c}(\Phi)+b_0$, with (\ref{Apotent})
then replaced by $B(\Phi)\ge b_0$.

%There are {\em many} possible potentials for $P$-designs.
There are {\em many} possible potentials for $P$-designs,
and so $P$-designs are ``universally optimal distributions of points'' (cf.\ \cite{CK07})
for the above class of such potentials.
Given a tractable formula for a potential,
$P$-designs 
can be constructed
numerically 
(for sufficiently many points) 
by minimising it (to zero) 
\cite{BGMPV21}, \cite{EW25}. We now investigate
the case of real and complex spherical designs in detail,
where $K(x,y)$ is a univariate function $F$ of $\inpro{x,y}$,
which we also refer to as (giving) a {\bf potential}, and hence $A_{w,c}(\Phi)$ is
indeed function of the Gramian of $\Phi$.
Potentials which depend on triples of points, and hence not the
Gramian in general, have recently been considered by \cite{BFGMPV22}.

%\vfil\eject

\section{Real spherical designs}
\label{realdesignsect}

%Let $\Harm_\ell(\Rd)$ be the space of homogeneous harmonic polynomials
%on $\Rd$ of degree $\ell$. We recall that the restriction map $f\mapsto f|_\SS$
%in injective on $\Harm_\ell(\Rd)$, and that spaces $\Harm_\ell(\Rd)$ 
%and $\Harm_\ell(\SS):=\Harm_\ell(\Rd)|_\SS$ are called the
%({\bf solid} and {\bf surface}) {\bf spherical harmonics} of degree $\ell$.
%Since the spherical harmonics are absolutely irreducible invariant
%subspaces of polynomials under the action of the unitary (orthogonal)
%group, which are orthogonal to each other, with
%$\Hom_m(\Rd)|_\SS = \Harm_m(\SS)\oplus\Harm_{m-2}(\SS)+\cdots$, 
%it follows that every unitarily invariant subspace of polynomials has the form

We have observed in (\ref{PLobservation}) that every  %finite dimensional 
unitarily invariant polynomial 
space $P$ restricted to the real sphere has the form 
\begin{equation}
\label{invariantPform}
P|_\SS=P_L:=\bigoplus_{\ell\in L} \Harm_\ell(\SS) \qquad
\hbox{(orthogonal direct sum)},
\end{equation}
with $L\subset\NN=\{0,1,2,\ldots\}$ an index set
(which is finite for $P|_\SS$ finite-dimensional). % a finite index set. 
For the irreducible subspace $\Harm_\ell(\Rd)$, the reproducing kernel is
%$$ K_d^{(\ell)}(x,y) 
%= \norm{x}^\ell\norm{y}^\ell Z_d^{(\ell)} \Bigl({\inpro{x,y}\over\norm{x}\norm{y}}\Bigl)
%:= \norm{x}^\ell\norm{y}^\ell
%\Bigl\{ C_\ell^{({d\over 2})}\Bigl({\inpro{x,y}\over\norm{x}\norm{y}}\Bigl)
%- C_{\ell-2}^{({d\over 2})}\Bigl({\inpro{x,y}\over\norm{x}\norm{y}}\Bigl)
%\Bigr\} , $$
\begin{equation}
\label{Kadditionform}
K_\ell^{(d)}(x,y) 
:= \norm{x}^\ell\norm{y}^\ell Q_\ell^{(d)}
\Bigl({\inpro{x,y}\over\norm{x}\norm{y}}\Bigl),
\end{equation}
where $Q_k=Q_{k}^{(d)}$ are the orthogonal polynomials for the
Gegenbauer weight for $\gl={d-2\over2}$, i.e., $(1-x^2)^{d-3\over2}$ on $[-1,1]$, with the 
normalisation $Q_{k}^{(d)}(1)=\dim(\Harm_k(\Rd))$ \cite{DGS77}. 
These satisfy
\begin{align}
Q_\ell^{(d)}(x)
& = (2\ell+d-2) \sum_{j=0}^{[\ell/2]}
(-1)^j {d(d+2)\cdots(d+2\ell-2j-4)\over 2^j j!(\ell-2j)!} x^{\ell-2j} 
\nonumber
\\
&= C_\ell^{({d\over 2})}(x) -C_{\ell-2}^{({d\over 2})}(x) 
= {2\ell+d-2\over d-2}  C_\ell^{({d-2\over2})}(x),
\label{Qkdefn}
%\quad \hbox{({Funk--Hecke formula})} \\
\end{align}
where $C_\ell^{(\gl)}$ are the Gegenbauer (ultraspherical) polynomials,
with $C_\ell^{(\gl)}:=0$, $\ell<0$.
They are given by $Q_0(x)=1$, $Q_1(x)=dx$, and the three-term recurrence
$$ \gl_{k+1}Q_{k+1}(x) = xQ_{k}(x)-(1-\gl_{k-1})Q_{k-1}(x), \qquad
\gl_k:={k\over2k+d-2}. $$
%where $Q_0(x):=1$, $Q_1(x):=dx$. 

Since surface area measure on the sphere is unitarily invariant, 
for $U$ unitary, 
we have
$$ (f\circ U)(x) = \int_\SS K(Ux,y)f(y)\,d\gs(y)
= \int_\SS K(Ux,Uy)(f\circ U)(y)\,d\gs(y). $$
%%
%%Since the action of the unitary group on the sphere is $2$-transitive, 
%%it follows that 
%\begin{itemize}
%\item The reproducing kernel $K(x,y)$ of
%a unitarily invariant space $P$ is a function of $\inpro{x,y}$,
%and hence the potential is a function of the Gramian $[\inpro{v_j,v_k}]$.
%\end{itemize}
Hence
the reproducing kernel $K(x,y)$ of
a unitarily invariant space $P$ is a function of $\inpro{x,y}$,
and hence the potential is a function of the Gramian $[\inpro{v_j,v_k}]$.
The direct calculation of the formula (\ref{Kadditionform})
is called the {\it addition formula} (see \cite{BHS19}). We also observe, that for the
real sphere
\begin{equation}
\label{realspherenorminpro}
%\norm{x-y}^2=\inpro{x-y,x-y}=\norm{x}^2+\norm{y}^2-2\norm{x,y} \Implies
\inpro{x,y}={1\over2}\bigl(\norm{x}^2+\norm{y}^2-\norm{x-y}^2\bigr)
=1-{1\over2}\norm{x-y}^2,
\end{equation}
so that the reproducing kernel on $\SS$ % the sphere 
can also be thought 
of as a function
of the ``squared distance'' $\norm{x-y}^2$, which %I believe 
is the direction generalised by Delsarte spaces (see \cite{H92}).

Since $\Harm_0(\SS)=\spam\{1\}$,
a potential for $P$ of the form (\ref{invariantPform}), 
with weights $w_j$, is given by
\begin{equation}
\label{AwFdef}
A_{w,c}(\Phi) = A_{F,w}(\Phi):= \sum_{j=1}^n \sum_{k=1}^n w_j w_k
F(\inpro{v_j,v_k}),
\end{equation}
where $F$ is the {\it univariate} polynomial given by
\begin{equation}
\label{FCtele}
F(x) 
:= \sum_{\ell\in L\setminus\{0\}} c_\ell Q_\ell(x)
= \sum_{\ell\in L\setminus\{0\}} c_\ell 
\bigl\{ C_\ell^{({d\over 2})}(x)
- C_{\ell-2}^{({d\over 2})}(x) \bigr\}.
\end{equation}
Indeed, every univariate polynomial $F=\sum_k c_k Q_k$ with $c_k\ge0$ gives 
a potential for the unitarily invariant polynomial
space $P=P_L$ of (\ref{invariantPform}), where $L:=\{k>0:c_k>0\}$.
We will say that such a polynomial $F$ is a {\bf potential}
for $P_L$. %, namely the $A_{F,w}(\Phi)$ of (\ref{AwFdef}).
The function $F$ of (\ref{FCtele}) is an example of what Schoenberg \cite{S42} calls
a {\bf positive definite function on the sphere}, i.e., a continuous function 
$F:[-1,1]\to\RR$ for which the 
right-hand side of (\ref{AwFdef})
$$ \sum_{j=1}^n \sum_{k=1}^n w_j w_k
F(\inpro{v_j,v_k}) = w^T [F(\inpro{v_j,v_k})]_{j,k=1}^n w $$
 is nonnegative for all choices of points $(v_j)\subset\SS(\Rd)$ 
and $w=(w_j)\in\RR^n$, $n=1,2,\ldots$. These are characterised by the fact that their Fourier 
series in $(Q_\ell)$ has only nonnegative coefficients \cite{DX13}.
In this terminology, we can paraphrase our observation:

\begin{itemize}
\item A positive definite function $F$ on the real sphere which is a polynomial is 
a potential for a spherical $P$-design, where $P=P_L$ is given by the 
correspondence (\ref{FCtele}).
\end{itemize}

We observe that the value of the potential $A_{F,w}(\Phi)$ depends only on
the set of angles %$A$ 
and their weighted multiplicities (see Section \ref{boundsection}), i.e.,
$$ \Ang(\Phi)=\{\inpro{v_j,v_k}\}_{1\le j,k\le n}, \qquad
m_{w,\ga}=\sum_{j,k\atop\inpro{v_j,v_k}=\ga}w_jw_k,\quad \ga\in A. $$

\begin{example}
\label{harmonicindextQ}
For spherical designs of harmonic index $m$ \cite{BOT15}, 
$P=\Harm_m(\Rd)$, i.e., $L=\{m\}$, and there is a unique (up to a scalar) potential $Q_m(x)$
given by (\ref{Qkdefn}).
\end{example}

\begin{example}
\label{realtightframes}
(Tight frames)
For $P=\Hom_2(\Rd)$, 
$P|_\SS %= \Harm_2(\SS)\oplus\Harm_0(\SS) 
= \Harm_2(\SS)\oplus\spam\{1\}$, i.e., $L=\{0,2\}$, and
$$ Q_2^{(d)}(x)= {1\over 2} d(d+2)\Bigl(x^2-{1\over d}\Bigr). $$
gives the canonical potential. The zeros of this potential are 
the unit norm tight frames, and if the $2$-weights $w_j=w_j^{(2)}$ given
by (\ref{wjdef}) are taken, then one obtains the variational characterisation
of tight frames \cite{W03} (take $m=2$ in (\ref{varcharweightedhalfdesign}) of 
	Theorem \ref{mweightsuffthm}), \cite{BF03}.
% Example \ref{halfdesignposweights}).
\end{example}

For spherical half-designs of order $m$, i.e.,
$$ P|_\SS=\Hom_m(\Rd)|_\SS=\Harm_m(\SS)\oplus\Harm_{m-2}(\SS)\oplus
\Harm_{m-4}(\SS)\oplus\cdots, \quad m\ge1, $$
the sum over $\ell$ for $c_\ell=1$
in (\ref{FCtele})
is telescoping, with $C_0^{({d\over2})}(x)=1$, and we obtain:
% giving the canonical potential

\begin{example} 
\label{canpotrealhalf}
The canonical potential for spherical half-designs of order $m$ is given by 
$$ F(x) =
\begin{cases}
C_m^{({d\over 2})}(x), & \hbox{$m$ odd}; \\
C_m^{({d\over 2})}(x) - 1, &\hbox{$m$ even},
\end{cases} $$
i.e.,
%\begin{equation}
%\label{canonpot}
%A_{w,1}(\Phi) =\sum_{j=1}^n \sum_{k=1}^n w_j w_k 
%\begin{cases}
%C_m^{({d\over 2})}(\inpro{v_j,v_k}) - 1, &\hbox{$m$ even}; \\
%C_m^{({d\over 2})}(\inpro{v_j,v_k}), & \hbox{$m$ odd}.
%\end{cases}
%\end{equation}
\begin{equation}
\label{canonpot}
A_{w,1}(\Phi) =\sum_{j=1}^n \sum_{k=1}^n w_j w_k C_m^{({d\over 2})}(\inpro{v_j,v_k})-\gep,
\qquad
\gep:=\begin{cases}
0 , & \hbox{$m$ odd}; \\
1, &\hbox{$m$ even}.
\end{cases}
\end{equation}
\end{example}
%\begin{align}
%\label{canonpot}
%A_{w,1}(\Phi) &=\sum_{j=1}^n \sum_{k=1}^n w_j w_k 
%\bigl\{ C_m^{({d\over 2})}(\inpro{v_j,v_k}) - 1 \bigr\} 
%\quad\hbox{($m$ even)}, \\
%A_{w,1}(\Phi) & = \sum_{j=1}^n \sum_{k=1}^n w_j w_k 
% C_m^{({d\over 2})}(\inpro{v_j,v_k})
%\quad\hbox{($m$ odd)}.
%\end{align}
Since the univariate polynomials $C_\ell^{({d\over2})}$ are even or odd,
with monomial coefficients of alternating sign, 
it turns out that we can choose the $c_\ell$ to
obtain a potential for the spherical half-designs with a very simple form. Let
\begin{equation}
\label{bmdefn}
b_m(\Rd) :=
\int_\SS\int_\SS \inpro{x,y}^m d\gs(x)\,d\gs(y)
=
\begin{cases}
0, & \hbox{$m$ odd}; \cr
{1\cdot3\cdot5\cdots(m-1)\over d(d+2)\cdots(d+m-2)}, & \hbox{$m$ even}.
\end{cases}
\end{equation}

\begin{lemma}
\label{halfdesignlemma}
(Half-designs) 
A potential for the spherical designs for $P=\Hom_m(\Rd)$,
equivalently $L=\{m,m-2,\ldots\}$,
is given by the polynomial $F(x)=x^m-b_m(\Rd)$, i.e.,
$$ A_{F,w}(\Phi) = \sum_{j=1}^n \sum_{k=1}^n w_jw_k 
%\Bigl\{ 
\inpro{v_j,v_k}^m - b_m(\Rd) %\Bigr\}
, $$
where $b_m(\Rd)$ is given by {\rm (\ref{bmdefn})}.
\end{lemma}

\begin{proof}
The index set is
$L=\{m-2a:1\le m-2a\le m\} = \{m-2a:0\le a\le {m-1\over2}\}$.
Let
$$ c_{m,a} := {1\over2^a a!} \prod_{r=1}^a (d+2(m-r)),
\qquad 0\le a\le {m-1\over2}. $$
Then a simple calculation, using the explicit formula of (\ref{Qkdefn}), gives
%$$ \sum_{0\le a\le {m-1\over2}} c_{m,a} Z_d^{(m-2a)} (x)
%= \hbox{${d(d+2)\cdots (d+2m-4)(d+2m-2) \over m!}$} x^m-c_0, $$
$$ F(x):= {m!\over d(d+2)\cdots (d+2m-2)} 
\sum_{0\le a\le {m-1\over2}} c_{m,a} Q_{d,m-2a} (x)
= x^m-b_m(\Rd), $$
so that we have a potential
$$ 
A_{F,w}(\Phi) = \sum_{j=1}^n \sum_{k=1}^n w_jw_k 
\Bigl\{ \inpro{v_j,v_k}^m - b_m(\Rd) \Bigr\}
= \sum_{j=1}^n \sum_{k=1}^n w_jw_k 
\inpro{v_j,v_k}^m - b_m(\Rd), $$
%where
%$$c_0:= \begin{cases}
% {(d+m)(d+m+2)\cdots(d+2m-2)\over 2\cdot 4\cdots (m-2)\cdot m}, & \hbox{$m$ even}; \cr
%0, &\hbox{$m$ odd}.  \end{cases} $$
as claimed.
\end{proof}

This result was proved for $m$ even by Venkov \cite{V01} (constant weights) and 
\cite{KP11} (nonnegative weights), by using a different method.

We can now describe various variational conditions for being
a real spherical $t$-design.

\begin{theorem}
\label{tdesignvarthm}
Let $w_1+\cdots+w_n=1$, $w_j\in\RR$, $v_j\in\SS$, and $t\ge1$.
Then $(v_j),(w_j)$ is a weighted
spherical $t$-design if and only if there is
equality in the inequalities
\begin{equation}
\label{weightedt-designcdn}
\sum_{j=1}^n \sum_{k=1}^n w_jw_k \inpro{v_j,v_k}^t \ge b_t(\Rd),
\qquad
\sum_{j=1}^n \sum_{k=1}^n w_jw_k \inpro{v_j,v_k}^{t-1} \ge b_{t-1}(\Rd).
\end{equation}
In particular, $(v_j)$ is a spherical $t$-design if and only if
\begin{equation}
\label{t-designcdn}
{1\over n^2} \sum_{j=1}^n \sum_{k=1}^n \inpro{v_j,v_k}^m 
=\int_\SS\int_\SS \inpro{x,y}^m\,d\gs(x)\,d\gs(y)=b_m(\Rd),
\qquad m\in\{t-1,t\}.
\end{equation}
\end{theorem}

\begin{proof}
Since $\Pi_t(\Rd)|_\SS=\Hom_t(\Rd)|_\SS\oplus\Hom_{t-1}(\Rd)|_\SS$, the spherical $t$-designs
are precisely the spherical half-designs of order $t$ which are also half-designs of order $t-1$,
and the result
follows directly from Lemma \ref{halfdesignlemma}.
\end{proof}

The condition (\ref{t-designcdn}) for $m\in\{1,2,\ldots,t\}$ is well known
(see \cite{GS79}, Theorem 4.4).
The following %corresponding 
condition given by the canonical potential (\ref{canonpot}) appears to be new.

\begin{corollary}  %\begin{example} 
\label{canonicalpotrealtdesigns}
The unit vectors $(v_j)\subset\Rd$ give a weighted
spherical $t$-design if and only if
\begin{equation}
\label{newtdesigncdn}
\sum_{j=1}^n\sum_{k=1}^n w_jw_k\, C_m^{({d\over2})}(\inpro{v_j,v_k})
=\begin{cases}
0, & \hbox{$m$ odd}; \cr
1, & \hbox{$m$ even},
\end{cases}
\qquad m\in\{t-1,t\},
\end{equation}
which can also be written as
\begin{equation}
\label{nicetdesigncdn}
\sum_{j=1}^n\sum_{k=1}^n w_jw_k\,\bigl\{
 C_t^{({d\over2})}(\inpro{v_j,v_k})+ C_{t-1}^{({d\over2})}(\inpro{v_j,v_k})
\bigr\}=1.
\end{equation}
\end{corollary}  %\end{example}

The condition that spherical $t$-designs are characterised by 
their canonical potentials for $\Harm_\ell(\Rd)$ vanishing, i.e.,
$$ \sum_{j=1}^n\sum_{k=1}^n w_jw_k\, C_\ell^{({d-2\over2})}(\inpro{v_j,v_k})=0, 
\qquad \ell\in\{1,2,\ldots,t\},$$
is well known, as is the generalisation to spherical designs of harmonic index $t$
\cite{ZBBJY17} (Lemma 2.1).

We now show that the necessary conditions (\ref{mweights1}) and (\ref{mweights2}),
which are given in Sidel'nikov \cite{Si74b} (Corollary 1) for unit vectors,
are sufficient.

\begin{theorem}
\label{mweightsuffthm}
Let $(v_j)$ be vectors in $\Rd$, not all zero, with corresponding $m$-weights, i.e.,
\begin{equation}
\label{wjdefn}
w_j:={\norm{v_j}^m\over\sum_\ell \norm{v_\ell}^m}. 
\end{equation}
Then $(v_j)$ gives a weighted spherical half-design of order $m$ 
if and only if there is equality in the inequality
\begin{equation}
\label{varcharweightedhalfdesign}
\sum_{j=1}^n \sum_{k=1}^n \inpro{v_j,v_k}^m 
\ge b_m(\Rd) \Bigl(\sum_{\ell=1}^n\norm{v_\ell}^m\Bigr)^2.
\end{equation}
Moreover, this is also a weighted spherical $m$-design if and only if
in addition there is equality in % the inequality
\begin{equation}
\label{varcharweightedtdesignpartII}
\sum_{j=1}^n \sum_{k=1}^n \norm{v_j}\norm{v_k} \inpro{v_j,v_k}^{m -1}
\ge b_{m-1}(\Rd) \Bigl(\sum_{\ell=1}^n\norm{v_\ell}^m\Bigr)^2.
\end{equation}
\end{theorem}

\begin{proof}
Take $w_j$ given by (\ref{mweightsuffthm}) in Lemma \ref{halfdesignlemma}
and Theorem \ref{tdesignvarthm}.
\end{proof}

%For positive weights, Lemma \ref{halfdesignlemma} can be expressed as follows.
%
%\begin{example} (Positive weights)
%\label{halfdesignposweights}
%For vectors $v_1,\ldots,v_n$ in $\Rd$, not all zero, 
%we have
%$$ \sum_{j=1}^n\sum_{k=1}^n\inpro{v_j,v_k}^m 
%\ge b_m(\Rd) \Bigl(\sum_{\ell=1}^n \norm{v_\ell}^m\Bigr)^2, $$
%with equality if and only if $({v_j\over\norm{v_j}})$ is a spherical half-design of
%order $m$ with weights (\ref{wjdef}).
%\end{example}

To find spherical $t$-designs with nonnegative weights, 
one can minimise the single potential
$$ A(\Phi):=\sum_{j=1}^n \sum_{k=1}^n \bigl( 
\inpro{v_j,v_k}^t 
+\norm{v_j}\norm{v_k}\inpro{v_j,v_k}^{t-1} \bigr)
\ge c_t(\Rd)+c_{t-1}(\Rd), $$
over the compact set of $\Phi=(v_j)$ with $\sum_{\ell}\norm{v_\ell}^m=1$,
where equality gives a weighted spherical $t$-design for the $t$-weights.
%where $A(\Phi)=b_t(\Rd)+b_{t-1}(\Rd)$ is the condition
%for being a spherical $t$-design.

If $X=\{v_j\}$ is antipodal (centrally symmetric), i.e., 
$X=-X$, then it gives a spherical half-design of odd order $m$,
for every $m$.
We will say a spherical half-design of odd order is {\bf nontrivial}
if its vectors span $\Rd$ and it is not antipodal.

\begin{example} The $d+1$ vertices of a regular simplex in $\Rd$ are a
nontrivial example of a spherical half-design of order $m=1$, via direct
calculation of (\ref{varcharweightedhalfdesign})
$$ \sum_j \sum_k \inpro{v_j,v_k}^1
= (d+1)+\{(d+1)^2-(d+1)\}\bigl({-1\over d}\bigr) = 0. $$
Hardin and Sloane \cite{HS96} give various half-designs for $\RR^3$,
e.g., they give $11$, $13$ and $15$-point spherical $3$-designs which
are nontrivial half-designs of order $3$.
% in Centrally symmetric $(v_j)$ are trivial examples (these
%integrate all odd polynomials). Hardin and Sloan give a spherical $3$-design of
%$11$ points for $\RR^3$ (which I checked does satisfy this).
\end{example}

It follows from the definition of designs (\ref{cuberule}), that a weighted 
spherical half-design of order $m$ is a weighted spherical-half design of
orders $m-2,m-4,\ldots$, with the same weights. 
Expressing this observation in terms of (\ref{varcharweightedhalfdesign}) gives:

\begin{example} Let $(v_j)$ be a sequence in $\Rd$ and $m$ be an odd
positive integer. If 
$$ \sum_{j=1}^n\sum_{k=1}^n \inpro{v_j,v_k}^m =0, $$
then 
\vskip-1truecm
$$ \sum_{j=1}^n\sum_{k=1}^n \norm{v_j}^{m-\ell}\norm{v_k}^{m-\ell}
 \inpro{v_j,v_k}^\ell =0, \qquad \ell=1,3,5,\ldots,m. $$
\end{example}

\begin{example} (Sharp configurations)
In \cite{CK07}, an $f$-potential energy for a finite set points $\cC$
on the real sphere by 
	$$ \sum_{x,y\in\cC\atop x\ne y} f(\norm{x-y}^2), $$
where $f:=(0,4]\to[0,\infty)$ is any decreasing continuous function.
In view of (\ref{realspherenorminpro}), this can be written as
$$ \sum_{x,y\in\cC\atop x\ne y} F(\inpro{x,y}), \qquad 
F(t):=f(2(1-t)), \quad -1\le t<1. $$
A subset $\cC$ of $\SS$ is a {\bf sharp configuration} (or {\bf code}) if there
are $m$ inner products between distinct points and it is a spherical
	$(2m-1)$-design. It is shown by Cohn and Kumar \cite{CK07} (also see
	\cite{BHS19} \S5.7) that if $\cC$ is a sharp configuration or
the vertices of the $600$-cell, and the above $f$ is 
%a positive definite function on the sphere which is 
completely monotonic, i.e., $(-1)^k f^{(k)}(t)\ge 0$, $k\ge0$, 
equivalently, $F$ is absolutely monotonic, i.e., $F^{(k)}(t)\ge 0$, $k\ge0$,
%(which precludes them being polynomials), 
then 
$$ \sum_{x,y\in\cC'\atop x\ne y} f(\norm{x-y}^2)
\ge \sum_{x,y\in\cC\atop x\ne y} f(\norm{x-y}^2), $$
for any other set of points with $|\cC'|=|\cC|$, i.e., 
	$\cC$ is a uniformly optimal distribution of points on the sphere
	(for all such $f$).
\end{example}

\section{Complex spherical designs}
\label{complexdesignsect}

The complex unitary matrices for $\Cd$ are a subgroup of the unitary matrices for $\RR^{2d}$
(the orthogonal group), and so irreducible subspaces under the action of the orthogonal group
may not be irreducible under the action of the complex unitary group.

For the complex sphere $\SS$, 
%$\Hom_k(\Cd)$ decomposes into unitarily invariant subspaces
%$$ \Hom_k(\Cd) = \Hom_{k,0}(\Cd)\oplus\Hom_{k-1,1}(\Cd)\oplus
%\cdots\oplus\Hom_{k,0}(\Cd), $$
the harmonic functions $\Harm_k(\Cd)\approx\Harm_k(\RR^{2d})$ can be further decomposed
into orthogonal (complex) unitarily invariant irreducible subspaces
$$ \Harm_k(\Cd) = H(k,0)\oplus H(k-1,1)\oplus\cdots\oplus H(0,k), $$
where $H(p,q)$ consists of all harmonic homogeneous polynomials on $\Cd$
that have degree $p$ in the variables $z_1,\ldots,z_d$ and
degree $q$ in the variables $\overline{z_1},\ldots,\overline{z_d}$
(see \cite{R80}).
Thus the unitarily invariant subspaces of polynomials restricted 
to the complex sphere have the form
$$ P|_\SS = P_\tau:= \bigoplus_{(p,q)\in\tau} H(p,q), \qquad
\hbox{(orthogonal direct sum)} $$
for $\tau$ a finite subset of indices from $\{(j,k):j,k\ge0\}$. 
Thus, the most general complex spherical design is one which 
integrates $P_\tau$, which we call a ({\bf spherical}) {\bf $\tau$-design}. 
Aspects of these $\tau$-designs
have been studied by \cite{MOP11}, \cite{RS14} %, \cite for $\tau$ a {\bf lower set}
%i.e., one satisfying $(k,l)\in\tau$
%whenever there is a $(p,q)\in\tau$ with $k\le p$, $l\le q$.
and \cite{MW23}.

There is a subtlety in defining classes of complex spherical designs, as the $\tau$ defining
a class is not unique (as is the $L$ in the real case), as we now see.

The reproducing kernel $K_{pq}=K_{pq}^{(d)}$ for $H(p,q)$ (and hence for any unitarily 
invariant space of polynomials)
has been calculated explicitly by \cite{F75} for $d>1$ as 
a function of $\inpro{z,w}$
$$K_d^{(p,q)}(z,w) = Q_{pq}^{(d)}(\inpro{z,w}), \qquad z,w\in\Cd, $$
where $Q_{pq}=Q_{pq}^{(d)}$
is the univariate polynomial
%$$ K_d^{(p,q)}(z,w) = H_d^{(p,q)}(\inpro{z,w}), $$
\begin{align}
Q_{pq}^{(d)}(z) &:= c_{pq}^{(d)} \,
z^{p-\min\{p,q\}}\overline{z}^{q-\min\{p,q\}}
%z^{p-q} |z|^{|p-q|-(p-q)}
{P^{(d-2,|p-q|)}_{\min\{p,q\}} (2|z|^2-1) \over 
P^{(d-2,|p-q|)}_{\min\{p,q\}} (1)} 
\nonumber \\
& = {p+q+d-1\over (d-1)!} \sum_{j=0}^{\min\{p,q\}} (-1)^j {(d+p+q-j-2)!\over j!(p-j)!(q-j)!}
z^{p-j}\overline{z}^{q-j},\quad z\in\CC, 
\label{Qpqdef}
\end{align}
and
\begin{equation}
\label{Hpqdimension}
c_{pq}^{(d)} := \dim(H(p,q))
= {(p+q+d-1)(p+d-2)!(q+d-2)!\over p!q!(d-1)!(d-2)!}
= Q_{pq}^{(d)}(1).
\end{equation}
The second of these formulas also holds for $d=1$, where
$H(p,q)=0$, unless $p=0$ or $q=0$, in which
case $H(p,0)=\spam\{z^p\}$, $H(0,q)=\spam\{\overline{z}^q\}$.
We also have % note the decomposition
$$ \Hom(p,q)=\Hom_{p,q}(\Cd)|_\SS = H(p,q)\oplus H(p-1,q-1)\oplus\cdots. $$
We observe that the expansion for $Q_{pq}(z)$ in terms of the 
monomials $z^j\overline{z}^k$ has real coefficients, so that
%$$\overline{Q_{pq}^{(d)}(z)} = Q_{pq}^{(d)}(\overline{z}) = Q_{qp}^{(d)}(z), $$ 
$$\overline{Q_{pq}(z)} = Q_{pq}(\overline{z}) = Q_{qp}(z), $$ 
and so the canonical potentials for $H(p,q)$ and $H(q,p)$ are equal,
by the calculation
\begin{align} 
\label{Hpqpots}
A_{w,H(p,q)}(\Phi) &= \overline{A_{w,H(p,q)}(\Phi)}
= \sum_j\sum_k w_j w_k \overline{Q_{pq}(\inpro{v_j,v_k})}  \cr
&= \sum_j\sum_k w_j w_k Q_{qp}^{(d)}(\inpro{v_j,v_k}) 
= A_{w,H(q,p)}(\Phi), 
\end{align}
as are those for $\Hom(p,q)$ and $\Hom(q,p)$.
%For $d=1$, we have $H(p,q)=0$, unless $p=0$ or $q=0$, in which
%case $H(p,0)=\spam\{z^p\}$, $H(0,q)=\spam\{\overline{z}^q\}$.
%$$ \Hom_{p,q}(\Cd) = H(p,q)\oplus H(p-1,q-1)\oplus\cdots. $$
Since the canonical potentials for $H(p,q)$ and $H(q,p)$ are equal,
the class of spherical designs for some unitarily invariant polynomial spaces 
are equal.
% are not in a $1$-$1$ correspondence
%with the finite subsets $L$.
For a set of indices $\tau$, we define 
$$ \tau^{\rm rev} := \{(q,p): (p,q)\in\tau\}. $$
By the {\bf class} of a spherical $\tau$-design we mean the maximal unitarily invariant
subspace that it integrates, or the indices $\tau^*$ of that subspace.

%all subspaces $H(p,q)$ that
%are integrated by all $\tau$-designs. Obviously, this includes all $H(p,q)$ for $(p,q)\in\tau$.

\begin{proposition} The class of complex spherical designs for 
$P_L$ and $P_K$ are equal if and only if
$$ L\cup L^\rev\cup\{0\} = K\cup K^\rev \cup \{0\}. $$
In particular, the class (of indices) for any $\tau$-design is
$$ \tau^*=\tau\cup\tau^\rev\cup\{0\}. $$
\end{proposition}

\begin{proof} Since every complex spherical design integrates the constants,
we can add $0=(0,0)$ to the set of indices $L$ without changing the class of
the spherical designs it gives. Similarly,
since $H(p,q)$ and $H(q,p)$ have the same canonical potential, by (\ref{Hpqpots}), 
we may add $(q,p)$ for $(p,q)\in L$. Thus if $L\cup L^\rev\cup\{0\} \subset K\cup K^\rev \cup \{0\}$, then a spherical design for $P_K$ is spherical design for $P_L$.
This gives the forward implication.

The converse follows from the fact that the canonical potentials for different classes
differ by at least one term, and then a linear algebra argument.
\end{proof}

We will not labour the point, but the classes of complex spherical designs are 
given by the possible choices for $\tau^*$, and for a $\tau$-design we will refer
to $\tau^*$ as the {\bf canonical indices} for the design, and use 
$\tau^*\setminus\{(0,0)\}$ %this 
to calculate the {\bf canonical potential}.

%As a particular example, the spherical designs for $\overline{P_L}=P_{L^\rev}$ 
%and for $P_L$ are same. 
%One could think of $(L\cup L^\rev)\setminus\{0\}$ as a canonical
%set of indices for this class of spherical designs, and define 
%a corresponding canonical potential (which is zero for the designs). 
%For the purpose of estimating
%the minimal number of vectors in such designs, it can be 
%to choose $L$ to is often best

%There is also a natural partial ordering of the 
%(canonical) subsets $L$ defining a class of spherical designs.
%Let $\cR=\cR^{\CC}$ be the operator
%$$ \cR \tau := \{(p-1,q-1):(p,q)\in\tau,\, p,q\ge1\}. $$

We will carry over terminology from the real case,
e.g., we say that a univariate polynomial
%$$ F=\sum_{(p,q)} c_{p,q} Q_{pq}^{(d)}, \qquad c_{p,q}\ge 0, $$
\begin{equation}
\label{complexF}
F=\sum_{(p,q)} f_{pq} Q_{pq}, \qquad f_{pq}\ge0,
\end{equation}
gives a potential $A_{F,w}(\Phi)$ for the unitarily invariant polynomial space
$$ P=P_\tau=\bigoplus_{(p,q)\in\tau} H(p,q), 
\qquad \tau:=\{(p,q): f_{pq}>0, (p,q)\ne(0,0)\}. $$

\begin{example} (Balanced sets) 
\label{Balancedsetexample}
Let $P_\tau=H(1,0)$, i.e., $\tau=\{(1,0)\}$.   %(0,1)$). 
Then 
	$$ \tau^*=\{(0,0),(1,0),(0,1)\}, \qquad Q_{10}^{(d)}(z)=dz, $$
so that the $\tau$-designs $(v_j)$ are characterised by 
$$ \sum_j\sum_k \inpro{v_j,v_k}
= \inpro{\sum_jv_j,\sum_k v_k}
= \Vert\sum_jv_j\Vert^2
= 0 \Implies \sum_jv_j=0, $$
i.e., the sum of their vectors is zero, and they are said to be
{\bf balanced}.
\end{example}

We have the complex version of Example \ref{realtightframes}.

\begin{example}
\label{complextightframeex}
 (Complex tight frames) 
For $P=\Hom_{1,1}(\Cd)$, 
$P|_\SS=H(1,1)\oplus H(0,0)$,  %$\dim(P|_\SS)=d^2$, and
i.e., $\tau=\tau^*=\{(1,1),(0,0)\}$, and
$$ Q_{11}^{(d)}(z) = d(d+1)\Bigl(|z|^2-{1\over d}\Bigr), $$
gives the canonical potential (which is zero for $d=1$) %is given by
$$ A_{P,w}(\Phi) = d(d+1)\sum_j\sum_k w_j w_k \Bigl(|\inpro{v_j,v_k}|^2-{1\over d}\Bigr). $$
The zeros of this potential are the unit norm tight frames for $\Cd$.
For vectors $(v_j)$ in $\Cd$, by taking the $2$-weights %$w_j=w_j^{(2)}$ 
given by (\ref{wjdef}),
one obtains the variational characterisation of tight frames \cite{W03}.
%by taking %weights % $2$-weights 
%$\displaystyle w_j={\norm{v_j}^2\over\sum_\ell\norm{v_\ell}^2}$.
\end{example}

The functions $F:\DD\to\CC$, on the complex unit disc $\DD=\{z\in\CC:|z|\le1\}$, 
of (\ref{complexF}) that give a potential, are in general complex-valued.
They do satisfy 
$$\overline{F(z)}=F(\overline{z}), $$
and so the potential $A_{F,w}$ 
that $F$ gives is real-valued, since we may group terms in (\ref{AwFdef})
$$ w_jw_k\bigl( F(\inpro{v_j,v_k}) + F(\inpro{v_k,v_j})\bigr)
= w_jw_k\bigl( F(\inpro{v_j,v_k}) + \overline{F(\inpro{v_j,v_k})}\bigr)
= 2 w_jw_k\Re\bigl(F(\inpro{v_j,v_k})\bigr). $$ % \in\RR. $$
It is possible to develop a theory of {\bf positive definite functions
on the complex sphere} which include such polynomials $F$ 
(see \cite{MP01}, \cite{MPP17}).

%Since  
%$H_d^{(2,0)}(z) = {1\over2} d(d+1) z^2$,
%$H_d^{(0,2)}(z) = H_d^{(2,0)}(\overline{z})
%{1\over2} d(d+1) \overline{z}^2$, and
%$$  H_d^{(2,0)}(z) = {1\over2} d(d+1) z^2, \qquad
%H_d^{(0,2)}(z) = H_d^{(2,0)}(\overline{z})
%{1\over2} d(d+1) \overline{z}^2. $$
The canonical potential 
%$$ A_{P,w}(\Phi)=\sum_j\sum_k w_jw_k F(\inpro{v_j,v_k}) $$
for 
$$ P=\Hom_2(\Cd), \qquad
P|_\SS = H(2,0)\oplus H(1,1)\oplus H(0,2)\oplus H(0,0), $$
is given by
$$ F(z)= Q_{20}^{(d)}(z)+Q_{11}^{(d)}(z)+Q_{20}^{(d)}(z)
= {1\over2} d(d+1)\Bigl( z^2 +2 z\overline{z}+\overline{z}^2
-{2\over d}\Bigr)
= 2d (d+1)\Bigl( \Re(z)^2 -{1\over 2d}\Bigr) . $$
Since this potential is zero if and only if $(v_j)$ is 
tight frame for $\RR^{2d}$, we have the following.

\begin{proposition}
Let $(v_j)$ be a sequence of $n$ vectors in $\Cd$, not all zero.
Then $(v_j)$ is tight frame for $\Cd$ if and only if there is
equality in the inequality
$$ \sum_j\sum_k|\inpro{v_j,v_k}|^2
\ge {1\over d}\Bigl(\sum_\ell\norm{v_\ell}^2 \Bigr)^2. $$
Moreover, it is also a tight frame for $\R^{2d}$ if and only if, 
in addition, there is equality in the inequality
$$ \sum_j\sum_k \inpro{v_j,v_k}^2 \ge 0. $$
\end{proposition}

\begin{proof} 
The first statement is just the variational characterisation of tight frames for 
$\Cd$ given in Example \ref{complextightframeex}.
For the vectors $(v_j)$ in $\Cd$ to be a tight frame for $\RR^{2d}$ the potentials
$A_{H(1,1)}(\Phi)$ and $A_{H(2,0)}(\Phi)=A_{H(0,2)}(\Phi)$ 
must be minimised for the $2$-weights. The first is minimised by being
a tight frame for $\Cd$, and the second if and only if there is equality
in 
%$$ A_{H(2,0)}(\Phi) = {1\over2}d(d+1) \sum_j \sum_k \inpro{v_j,v_k}^2 \ge 0. $$
$$ A_{H(2,0)}(\Phi)={1\over2}d(d+1) \sum_j \sum_k {\norm{v_j}^2\over C}{\norm{v_k}^2\over C} 
\inpro{{v_j\over\norm{v_j}},{v_k\over\norm{v_k}}}^2
={d(d+1)\over 2C^2} \sum_j \sum_k \inpro{v_j,v_k}^2\ge0, $$
%\begin{align*}
%A_{H(2,0)}(\Phi) &= {1\over2}d(d+1) \sum_j \sum_k {\norm{v_j}^2\over C}{\norm{v_k}^2\over C} 
%\inpro{{v_j\over\norm{v_j}},{v_k\over\norm{v_k}}}^2 \cr
%&={d(d+1)\over 2C^2} \sum_j \sum_k \inpro{v_j,v_k}^2\ge0,
%\qquad C=\sum_\ell \norm{v_\ell}^2>0.
%\end{align*}
where $C=\sum_\ell \norm{v_\ell}^2>0$.
%\vskip-1.0truecm
\end{proof}

This result appears in \cite{W25} (Theorem 3.1) where it is 
obtained in a similar way.
%from the canonical potentials for real and complex tight frames.

\begin{example} For $v_j=z_j=x_j+iy_j\in\CC$, not all zero, $(v_j)$ is
a tight frame for $\CC$ (since the first inequality holds). 
The condition for the vectors $(x_j,y_j)$ to be a tight frame for $\RR^2$
is that
$$ \sum_j\sum_k (z_j\overline{z_k})^2=
\Bigl(\sum_j z_j^2\Bigr)\Bigl(\sum_k \overline{z_k}^2\Bigr)
={\Bigl\lvert}\sum_j z_j^2{\Bigr\rvert}^2=0, $$
i.e., $\sum_j z_j^2=0$. This characterisation of tight frames for $\RR^2$
(in polar form) is given in \cite{Fic01} % Matt's thesis, 
and \cite{HKLW07} (where the vectors $z_j^2$ are called diagram vectors).
\end{example}

The potentials for $P=\Hom(2,2)=H(2,2)\oplus H(1,1)\oplus H(0,0)$ 
are given by
\begin{align*}
F(z) &= c_1 Q_{11}^{(d)}(z)+c_2 Q_{22}^{(d)}(z) \cr
&= c_1 (d+1)(d|z|^2-1)+c_2
{1\over 4} d(d+3)\bigl( (d^2+3d+2)|z|^4-4(d+1)|z|^2+2\bigr).
\end{align*}
The $F(z)$ for the canonical potential ($c_1=c_2=1$) has nonzero terms in $1$, $|z|^2$, $|z|^4$.
The term in $|z|^2$ can be cancelled by choosing
$c_1=(d+3)c_2$, which gives
$$ F(z) = |z|^4-{2\over d(d+1)}. $$
We now seek a similar `telescoping' sum for a general 
$\Hom(p,q)$, to obtain an analogue of Lemma \ref{halfdesignlemma}.

\begin{lemma} A direct calculation gives
\begin{equation}
\label{Hpqtele}
\sum_{j=0}^{\min\{p,q\}} {(d-1)!p!q!\over j!(p+q+d-1-j)!} Q_{p-j,q-j}^{(d)}(z)
= z^p\overline{z}^q.
\end{equation}
\end{lemma}

The complex analogue of (\ref{bmdefn}) is
\begin{equation}
\label{bqmdefn}
b_{p,q}(\Cd) 
:= \int_\SS\int_\SS \inpro{z,w}^p\overline{\inpro{z,w}}^q\, d\gs(z)\,d\gs(w)
%= \int_\SS\int_\SS \inpro{z,e_1}^p\overline{\inpro{z,e_1}}^q\, d\gs(z)\,d\gs(w)
%= \int_\SS z_1^p\overline{z_1}^q\, d\gs(z)
= \begin{cases}
0, & p\ne q;\cr
{(d-1)!p!\over(d-1+p)!}, & p=q.
\end{cases}
\end{equation}

\begin{theorem}
\label{HpqF}
A potential for $\Hom(p,q)=\Hom_{p,q}(\Cd)$ is given by 
\begin{equation}
\label{HompqF}
F(z)= z^p\overline{z}^q-b_{p,q}(\Cd), 
\end{equation}
where $b_{p,q}(\Cd)$ is given by (\ref{bqmdefn}),
and a potential for $\Hom_m(\Cd)$ is given by 
%(using the binomial identity on the above)
\begin{equation}
\label{HomppF}
F(z) = (z+\overline{z})^m - 2^m b_m(\RR^{2d})
= 2^m\bigl\{ \Re(z)^m- b_m(\RR^{2d})\bigr\}.
\end{equation}
where $b_m(\RR^{2d})$ is given by (\ref{bmdefn}).
\end{theorem}

\begin{proof} We note that all the coefficients in the expansion 
(\ref{Hpqtele}) are positive. Hence,
for $p\ne q$, this gives a potential
for $\Hom(p,q)=\oplus_j H(p-j,q-j)$, %\perp H(0,0)$, 
and for $p=q$, this is a potential for $\Hom(p,p)$ plus the $Q_{00}^{(d)}(z)$ constant term
$$ {(d-1)!p!p!\over p!(p+d-1)!} Q_{00}^{(d)}(z)
= {(d-1)!p!\over(p+d-1)!} = {1\over{p+d-1\choose p-1}} = b_{p,p}(\Cd). $$
This gives (\ref{HompqF}) in both cases.

A potential for $\Hom_m(\Cd)=\bigoplus_j \Hom(j,m-j)$ is given by
$$ F(z) 
= \sum_{j=0}^m {m\choose j} \bigl(z^j\overline{z}^{m-j}-b_{j,m-j}(\Cd)\bigr)
= (z+\overline{z})^m -\sum_{j=0}^m {m\choose j} b_{j,m-j}(\Cd). $$
The constant term subtracted above is zero, unless $m$ is even, in which case it is
$$ {m\choose{m\over 2}} b_{{m\over2},{m\over2}}(\Cd)
= {m!\over({m\over2})!({m\over2})!} {(d-1)!({m\over2})!\over(d-1+{m\over2})!}
= 2^m {m!\over2^{m\over2} ({m\over2})!} {(d-1)!\over 2^{m\over2}(d-1+{m\over2})!}
= 2^m b_m(\RR^{2d}).
$$
Hence we obtain (\ref{HomppF}). 
\end{proof}

%This we recognise as the potential for $\Hom_m(\RR^{2d})$ given in Lemma \ref{halfdesignlemma}.

The nonnegativity of the potential $A_{F,w}(\Phi)$ given by (\ref{HompqF}) is given 
as Lemma 3.3 in \cite{RS14}, with equality asserted when it holds
for all $(p,q)$ in a lower set $\tau$ for a unit-norm $\tau$-design $\Phi$. 
The complex spherical designs for the special case $\Hom(p,p)$ are ``projective designs'', 
which we will discuss in the next section.

\begin{example}
A potential for the holomorphic 
polynomials $\Hom(k,0)=H(k,0)$ is given by $F(z)=z^k$,
and for the holomorphic polynomials of degree $\le k$ a 
potential is given by any linear combination of $z,z^2,\ldots,z^k$
with positive coefficients.
\end{example}

We now give the weighted version of Theorem \ref{HpqF}.

\begin{theorem} 
\label{weightedversionHompq}
%(Weighted version) 
Let $m=p+q$. For any vectors $v_1,\ldots,v_n$ in $\Cd$, not all zero, we have
\begin{equation}
\label{Hom(p,q)weightedvarineq}
\sum_{j=1}^n\sum_{k=1}^n \inpro{v_j,v_k}^p\overline{\inpro{v_j,v_k}}^q
\ge b_{p,q}(\Cd)\Bigl(\sum_{\ell=1}^n \norm{v_\ell}^{p+q}\Bigr)^2,
\end{equation}
with equality if and only $(v_j)$ is an $m$-weighted spherical 
design for $\Hom(p,q)$, and
\begin{equation}
\label{Hom(m)weightedvarineq}
\sum_{j=1}^n\sum_{k=1}^n (\Re\inpro{v_j,v_k})^m
\ge b_m(\RR^{2d})\Bigl(\sum_{\ell=1}^n \norm{v_\ell}^{m}\Bigr)^2,
\end{equation}
with equality if and only if $(v_j)$ is an $m$-weighted spherical half-design of 
order $m$.
\end{theorem}

\begin{proof} Let $m=p+q$, and $C=\sum_\ell\norm{v_\ell}^m>0$. 
Then the potential given by (\ref{HompqF}) for the $m$-weights is
\begin{align*}
A_{F,w}(\Phi)
&=\sum_j\sum_k {\norm{v_j}^m\norm{v_k}^m\over C^2}
\Bigl\{ \inpro{{v_j\over\norm{v_j}},{v_k\over\norm{v_k}}}^p
\overline{\inpro{{v_j\over\norm{v_j}},{v_k\over\norm{v_k}}}}^q-b_{p,q}(\Cd) 
\Bigl\} \cr
&= {1\over C^2}\sum_j\sum_k\inpro{v_j,v_k}^p\overline{\inpro{v_j,v_k}}^q-b_{p,q}(\Cd).
\end{align*}
Multiplying this by $C^2$ and rearranging gives (\ref{Hom(p,q)weightedvarineq}).

For $P=\Hom(m)=\Hom_m(\Cd)$
a similar argument using (\ref{HomppF}) gives (\ref{Hom(m)weightedvarineq}).
\end{proof}

We will refer to spherical designs for $\Hom(p,q)$ as {\bf spherical $(p,q)$-designs}, 
which generalises the definition of complex spherical $(t,t)$-designs. We now show that 
the canonical potential for these has a simple form.

\begin{lemma}
\label{complexGegsumslemma}
The complex Gegenbauer polynomials satisfy
\begin{equation}
\label{QPqsumformula}
\sum_{j=0}^{\min\{p,q\}} Q_{p-j,q-j}^{(d)}(z) = {d\over p+q+d} Q_{pq}^{(d+1)}(z),
\end{equation}
equivalently,
\begin{equation}
\label{QPqdiff}
Q_{p,q}^{(d)}(z) = {d\over p+q+d} Q_{pq}^{(d+1)}(z) - {d\over p+q+d-2} Q_{p-1,q-1}^{(d+1)}(z).
\end{equation}
\end{lemma}

\begin{proof}
The equivalence of (\ref{QPqsumformula}) and
(\ref{QPqdiff}) is obvious. Thus it suffices to prove (\ref{QPqdiff}) by direct calculation
from (\ref{Qpqdef}), i.e., by equating coefficients.
%coefficients of $z^p\overline{z}^q$ are 
%$$ {(p+q+d-1)\over (d-1)!}{(d+p+d-2)!\over p!q!}
%=  {d\over p+q+d} {(p+q+d)\over d!}{(d+p+d-1)!\over p!q!} -0. $$
% Let $m=\min\{p,q\}$. 
%If $m=0$, then we have
%$$  \sum_{j=0}^{\min\{p,q\}} Q_{p-j,q-j}^{(d)}(z)
%=Q_{p,q}^{(d)}(z) ={(d+p+q-1)!\over (d-1)!p!q!} z^p\overline{z}^q 
%= {d\over p+q+d} Q_{pq}^{(d+1)}(z).$$
\end{proof}

From Lemma \ref{complexGegsumslemma}, we have

\begin{example}
\label{pq-designcanonpot}
(Spherical $(p,q)$-designs) 
The canonical potential for spherical $(p,q)$-designs,
i.e., $P=\Hom(p,q)=\Hom_{p,q}(\Cd)$ is given by
$$ \sum_{j=0}^{\min\{p,q\}} Q_{p-j,q-j}^{(d)}(z) = {d\over p+q+d} Q_{pq}^{(d+1)}(z). $$
\end{example}

\section{Projective spherical designs}
\label{projectivedesignsect}

We will say that a (weighted) spherical design $(v_j)$ for the polynomial space $P$
is a {\bf projective spherical design} if $(c_jv_j)$ is a spherical design
for all choices of unit scalars $c_j$. Clearly such a design 
can be thought of as a sequence of lines.
Since a projective design depends on the $(v_j)$ up to unitary equivalence and 
multiplication by unit scalars, 
it follows from \cite{CW16} 
that it can be characterised in terms of its $m$-products
$$ \gD(v_{j_1},\ldots,v_{j_m}):=\inpro{v_{j_1},v_{j_2}}\inpro{v_{j_2},v_{j_3}}
\cdots\inpro{v_{j_m},v_{j_1}}, \qquad 1\le j_1,\ldots,j_m\le n. $$
%We observe
%that %which of 
Hence
the reproducing kernels $K(x,y)$ which are invariant under 
this equivalence, i.e., replacing $(x,y)$ with $(c_x Ux, c_y Uy)$, 
where $U$ is unitary and $c_x,c_y$ are unit scalars,
are
those which are functions of $\inpro{x,y}^2$ and $|\inpro{x,y}|^2$,
respectively. Thus the polynomial spaces
giving projective spherical designs in the real and complex cases are 
\begin{itemize}
\item
The spaces $P_L$ consisting of even 
polynomials on $\Rd$, i.e., $L\subset\{0,2,4,\ldots\}$.
\item 
The spaces $P_\tau$ of polynomials on $\Cd$, where $\tau=\tau^*\subset\{(0,0),(1,1),\ldots\}$.
\end{itemize}
This notion of projective designs generalises %can be generalised 
to designs on the Grassmannian \cite{EG19}.

It follows from the multiplication rules for Gegenbauer polynomials
(Theorem \ref{Qpqprodtheorem}), that the univariate functions $F$ giving 
potentials for projective designs are closed under multiplication.

The projective spherical designs for the (projectively unitarily invariant) %polynomial 
spaces
$$ P_L = \Hom(2t)= \Harm(0) \oplus \Harm(2) \oplus\cdots\oplus\Harm(2t), 
\quad L={0,2,\ldots,2t}, $$
$$ P_\tau =\Hom(t,t)=H(0,0)\oplus H(1,1)\oplus\cdots\oplus H(t,t),
\quad \tau=\{(0,0),(1,1),\ldots,(t,t)\}, $$
are of particular interest. 
In the real case, these are the {\bf spherical half-designs} of order $2t$ 
\cite{KP11} and the real {\bf spherical $(t,t)$-designs} \cite{W17}.
In the complex case,
they are known as 
{\bf projective $t$-designs} on Delsarte spaces \cite{H90},\cite{W20},
complex {\bf spherical semi-designs} of order $2t$ \cite{KP17},
and complex {\bf spherical $(t,t)$-designs} \cite{W17}.
The potentials of Lemma \ref{halfdesignlemma} and %   (\ref{HompqF}) of 
Theorem \ref{HpqF} for these are
$$ F(x)=x^{2t}-b_{2t}(\Rd), \qquad
F(z)=|z|^{2t}-b_{t,t}(\Cd). $$
We may combine the characterisations of Theorem \ref{mweightsuffthm}
and Theorem \ref{weightedversionHompq} to obtain

\begin{example}
\label{(t,t)-designs}
For any vectors $v_1,\ldots,v_n$ in $\Fd$, not all zero, we have
$$ \sum_{j=1}^n\sum_{k=1}^n |\inpro{v_j,v_k}|^{2t}
\ge c_t(\Fd)\Bigl(\sum_{\ell=1}^n \norm{v_\ell}^{2t}\Bigr)^2, $$
where
$$ c_t(\Rd):=b_{2t}(\Rd)={1\cdot3\cdot5\cdots(2t-1)\over d(d+2)\cdots(d+2t-2)}, \qquad
c_t(\Cd):=b_{t,t}(\Cd)={1\over{t+d-1\choose t}}, $$
with equality if and only if $(v_j)\subset\Fd$ is a spherical $(t,t)$-design for $\Fd$.
\end{example}
This characterisation is given in \cite{KP11},\cite{KP17} and \cite{W17},
and the inequalities were first given by Sidel'nikov \cite{Si74b} and Welch \cite{W74}.

In Table \ref{P-designlist}, we summarise our calculations of potentials from 
Sections \ref{realdesignsect}, \ref{complexdesignsect} and \ref{projectivedesignsect}.

\setlength{\tabcolsep}{5pt}
\renewcommand{\arraystretch}{1.5}

%\begin{table}[!ht]
\begin{table}
\fontsize{10pt}{10pt}\selectfont
\caption{\small Selected real and complex spherical $P$-designs and their potentials. 
The inclusion of zero or not in the index set is for aesthetics. 
Those marked with a * are canonical potentials.  }
\begin{center}
\label{P-designlist}       % Give a unique label
\begin{tabular}{| >{$}l<{$} |>{$}l<{$} | >{$}l<{$} | >{$}l<{$} | l | }
%\begin{tabular}{| >{$}l<{$} |>{$}l<{$} | >{$}l<{$} | >{$}l<{$} | >{$}l<{$} | }
\hline
\hbox{$P$-design} & \hbox{$P$ and the index set $L$ or $\tau$} & \hbox{Potential $F$} & \hbox{Comments} \\
\hline
\hbox{harmonic index $m$} & \Harm_m(\Rd) & Q_m^{(d)}(x) & \hbox{Example \ref{harmonicindextQ} *} \\
& \{m\} && \\
\hbox{real tight frame} & \Hom_2(\Rd) & x^2-{1\over d} & \hbox{Example \ref{realtightframes} *} \\
& \{2\} && \\
\hbox{real half-design} & \Hom_m(\Rd) & x^m-b_m(\Rd) & \hbox{Lemma \ref{halfdesignlemma}} \\ 
& \{m,m-2,\ldots\} & C_m^{({d\over2})}(x)-1,\ \hbox{$m$ even} & \hbox{Example \ref{canpotrealhalf} *} \\ 
&& C_m^{({d\over2})}(x),\ \hbox{$m$ odd} & \\ 
\hbox{spherical $t$-design} & \Pi_t(\Rd) & 
C_t^{({d\over2})}(x)+C_{t-1}^{({d\over2})}(x)-1 & 
\hbox{Corollary \ref{canonicalpotrealtdesigns} *} \\
& \{1,2,\ldots,t\} & 
\sum\limits_{m=t,t-1}c_m\bigl(x^m-b_m(\Rd)\bigr) & \hbox{Theorem \ref{tdesignvarthm} } \\ 
\hbox{harmonic index $(p,q)$} & H(p,q) & Q_{pq}^{(d)}(z) & \hbox{Equation (\ref{Qpqdef}) *} \\
& \{(p,q)\} && \\
\hbox{complex tight frame} & \Hom_{1,1}(\Cd) & \hbox{$|z|^2-{1\over d}$} & \hbox{Example \ref{complextightframeex} *} \\
& \{(1,1)\} && \\
\hbox{spherical $(p,q)$-design} & \Hom_{p,q}(\Cd) & z^p\overline{z}^q-b_{p,q}(\Cd) & \hbox{Theorem \ref{HpqF}} \\
& \{(p,q),(p-1,q-1),\ldots\} & {d\over p+q+d} Q_{pq}^{(d+1)}(z), \quad p\ne q & \hbox{Example \ref{pq-designcanonpot} *}  \\
\hbox{spherical $(t,t)$-design} & \Hom_{t,t}(\Cd) 
& \hbox{$|z|^{2t}-c_t(\Cd)$} & \hbox{Example \ref{(t,t)-designs}} \\
& \{(1,1),\ldots(t,t)\} && \\
\hbox{complex half-design} & \Hom_m(\Cd) & \Re(z)^m - b_m(\RR^{2d}) & \hbox{Theorem \ref{HpqF}} \\
& \{(p,q): p+q=m,m-2,\ldots\} && \\
\hbox{spherical $t$-design} & \Pi_t(\Cd) & \sum\limits_{m=t-1,t} c_m\bigl(\Re(z)^m - b_m(\RR^{2d})\bigr) & \hbox{Theorem \ref{HpqF}} \\
& \{(p,q): p+q=1,2,\ldots,t\} && \\
\hline
\end{tabular}
\end{center}
\end{table}

\vfil\eject

\section{Orthogonality of the Gegenbauer polynomials}
\label{Gegenorthsect}

The polynomials $Q_k^{(d)}$ of (\ref{Qkdefn}) that give the reproducing kernel for
$\Harm_k(\Rd)$ are orthogonal with respect to the Gegenbauer
weight $(1-x^2)^{{d-2\over2}-{1\over2}}$ on $[-1,1]$.
Indeed
\begin{equation}
\label{realGegorthog}
\inpro{Q_j,Q_k}_\geg 
= \begin{cases}
Q_j(1)=\dim(\Harm_k(\Rd)), & j=k; \cr
0, & j\ne k,
\end{cases}
%\inpro{Q_j,Q_k}_\geg = Q_j(1)=\dim(\Harm_k(\Rd)), \quad j=k, \qquad
%\inpro{Q_j,Q_k}_\geg = 0, \quad j\ne k,
\end{equation}
where
$$ \inpro{f,g}_\geg 
:= 
{\gG({1\over2}d)\over\sqrt{\pi}\gG({1\over2}d-{1\over2})}
 \int_{-1}^1 f(x)g(x)\, (1-x^2)^{d-3\over2}\, dx. $$
We will refer to these, and the polynomials $Q_{pq}^{(d)}$ of (\ref{Qpqdef}) 
giving the reproducing kernel for $H(p,q)$ as (real or complex) {\bf Gegenbauer polynomials}.
The terms ``Jacobi'' and ``disk'' polynomial are sometimes used for the 
latter \cite{RS14}, \cite{MOP11}. 

We now show the polynomials $Q_{pq}^{(d)}$ %the complex Gegenbauer polynomials %these 
are orthogonal with respect to 
the Gegenbauer weight $(1-|z|^2)^{d-2}$ on the unit disc
$\DD=\{z\in\CC:|z|\le1\}$ in $\CC$ (cf.\ \cite{MOP11}).
%This is a new result.

\begin{proposition} For $d>1$, the (complex) Gegenbauer polynomials 
$Q_{pq}^{(d)}$ of {\rm (\ref{Qpqdef})} are orthogonal 
with respect to the inner product
% Gegenbauer weight $(1-|z|^2)^{d-2}$ on the unit disc 
%$\DD=\{z\in\CC:|z|\le1\}$ in $\CC$.
%$$\inpro{f,g}_\cgeg :={1\over\pi}\int_\DD f(z)\overline{g(z)}(1-|z|^2)^{d-2} %\, dA(z)
%:= {1\over\pi} \int_0^1\int_0^{2\pi} f(re^{i\gth})
%\overline{g(re^{i\gth})} \,(1-r^2)^{d-2} r\,dr\,d\gth. $$
\begin{align*}
\inpro{f,g}_\cgeg 
& :={d-1\over\pi}\int_\DD f(z)\overline{g(z)}(1-|z|^2)^{d-2} \, dA(z)  \cr
&= {d-1\over\pi} \int_0^1\int_0^{2\pi} f(re^{i\gth})
\overline{g(re^{i\gth})} \,(1-r^2)^{d-2} r\,dr\,d\gth,
\end{align*}
where $A$ is the area measure on $\CC=\RR^2$, and
\begin{equation}
\label{complexGegorthog}
\inpro{Q_{pq},Q_{k\ell}}_\cgeg =
\begin{cases}
Q_{pq}(1)=\dim(H(p,q)), & (p,q)=(k,\ell); \cr
0, & (p,q)\ne(k,\ell). 
\end{cases}
\end{equation}
\end{proposition}

\begin{proof} We write $z=re^{i\gth}$, so that 
%the inner product is
%$$ \inpro{f,g}_\cgeg = {1\over 2\pi} \int_0^1\int_0^{2\pi} f(re^{i\gth})
%\overline{g(re^{i\gth})} \,(1-r^2)^{d-2} r\,dr\,d\gth, $$ and
$$ Q_{pq}^{(d)}(z) 
%= c_d^{(p,q)} \, z^{p-\min\{p,q\}}\overline{z}^{q-\min\{p,q\}} {P^{(d-2,|p-q|)}_{\min\{p,q\}} (2|z|^2-1) \over P^{(d-2,|p-q|)}_{\min\{p,q\}} (1)}
= c_{pq}^{(d)} \, (re^{i\theta})^{p-m}(r\overline{e^{i\theta}})^{q-m}
{P^{(d-2,|p-q|)}_{m} (2r^2-1) \over P^{(d-2,|p-q|)}_{m} (1)}, 
\qquad m:=\min\{p,q\}. $$
By integrating in $\gth$ first, we see that 
$\inpro{Q_{pq}^{(d)},Q_{p'q'}^{(d)}}_\cgeg=0$, except when $p-q = p'-q'$. 
%$H_d^{(p,q)}$ and $H_d^{(p',q')}$ are orthogonal, unless $p-q = p'-q'$. 
%$$ p-m+q'-m' = q-m+p'-m' \Iff p+q' = q+p' \Iff p-q = p'-q'. $$
In this case, we may suppose, without loss of generality, that
$p-q = p'-q'=k\ge0$, i.e., $m=q$, $m'=q'$,
 so that 
$(re^{i\theta})^{p-m}(r\overline{e^{i\theta}})^{q-m}
(r\overline{e^{i\theta}})^{p'-m'}(re^{i\theta})^{q'-m'}
=r^{2k}$,
and 
%$p-m+q'-m'=p-m=(p-q)+(q-m)=k$. Then
$$ \inpro{Q_{pq}^{(d)},Q_{p'q'}^{(d)}}_\cgeg 
=  2 c_{pq}^{(d)}c_{p'q'}^{(d)}
\int_0^1 r^{2k}
P_m^{(d-2,k)}(2r^2-1)
P_{m'}^{(d-2,k)}(2r^2-1) \, (1-r^2)^{d-2}\, r\,dr. $$
By making the change of variables
%$$ x=2r^2-1, \quad dx=4r\, dr, \quad 1-r^2={1-x\over2}, \quad r^2={1+x\over2}. $$
$x=2r^2-1$, %$dx=4r\,dr$, $1-r^2={1-x\over2}$, $r^2={1+x\over2}$, 
the integral in $r$ above becomes
$$ \int_{-1}^1 \left({1+x\over2}\right)^{k}
P_m^{(d-2,k)}(x)
P_{m'}^{(d-2,k)}(x) \, \left({1-x\over2}\right)^{d-2}\, {1\over 4} dx, $$
which is zero, unless $m=m'$, in which case $(p,q)=(p',q')$.
The calculation of the constant for the nonzero inner product
is a straight forward calculation.
\end{proof}

This orthogonality of the polynomials $Q_{pq}^{(d)}$ can also be proved 
from the reproducing kernel property, 
by using the orthogonality of the $H(p,q)$, and the 
result (see \S1.4.5 of \cite{R80}) that for $\SS$ the unit ball in $\Cd$
and $f$ a univariate function
$$ \int_\SS f(\inpro{z,w})\,d\gs(w) = {d-1\over\pi} \int_\DD 
f(\zeta)(1-|\zeta|^2)^{d-2} \,dA(\zeta), \qquad z\in\Cd,\ \norm{z}=1. $$
In particular, we can calculate the orthogonality constant
\begin{align*} 
Q_{pq}^{(d)}(1)
= \overline{Q_{pq}^{(d)}(\inpro{z,z})}
&= \int_\SS  Q_{pq}^{(d)}(\inpro{z,w})  \overline{Q_{pq}^{(d)}(\inpro{z,w})} \,
d\gs(w) \cr
&= {d-1\over\pi} \int_\DD 
| Q_{pq}^{(d)}(\zeta)|^2(1-|\zeta|^2)^{d-2} \,dA(\zeta)
= \inpro{Q_{pq}^{(d)},Q_{pq}^{(d)}}_\cgeg.
\end{align*}

Since $\displaystyle \Harm_k(\Cd)=\bigoplus_{p+q=k} H(p,q)$, we have
$$ Q_k^{(2d)}\bigl((x,y)\bigr) = \sum_{p+q=k} Q_{pq}^{(d)}(z), \qquad z=x+iy\in\Cd. $$
For $d=1$, the polynomials $Q_{pq}^{(1)}(z)$ can be viewed 
as being orthogonal with respect to the ``singular Gegenbauer weight''
$$ \inpro{f,g}_\cgeg = {1\over 2\pi} \int_0^{2\pi} f(e^{i\gth})
\overline{g(e^{i\gth})} \,d\gth. $$

%The above theory extends to unitarily invariant spaces of
%polynomials on the complex sphere. We now provide details.
%With $Q_{pq}:=Q_{pq}^{(d)}$, we have the orthogonality relations
%$$ \inpro{Q_{pq},Q_{k\ell}}_\cgeg =
%\begin{cases}
%Q_{pq}(1)=\dim(H(p,q)), & (p,q)=(k,\ell); \cr
%0, & (p,q)\ne(k,\ell). 
%\end{cases} $$
%where
%$$ \inpro{f,g}_\cgeg := {d-1\over\pi} \int_0^1\int_0^{2\pi} f(re^{i\gth})
%\overline{g(re^{i\gth})} \,(1-r^2)^{d-2} r\,dr\,d\gth. $$

\section{Products of the Gegenbauer polynomials}
\label{Gegenprods}

To construct positive functions $F$ giving a potential for
a spherical design, it is useful to know the Gegenbauer
expansion for a product of Gegenbauer polynomials.
In the real case, there is the following celebrated formula 
dating back to Rogers (1895)
%
%Since
%$$ Q_m = {2m+2\nu\over 2\nu} C_m^\nu = {m+\nu\over \nu} C_m^\nu,
% \qquad\nu:={d-2\over 2}. $$
%we obtain from this
%$$ {m+\nu\over \nu} C_m^\nu {n+\nu\over \nu} C_n^\nu 
%= {m+\nu\over \nu} {n+\nu\over \nu} \sum_{j=0}^{\min\{m,n\}} 
%{\nu\over(m+n+\nu-j)} {(\nu)_j(\nu)_{m-j}(\nu)_{n-j}(2\nu)_{m+n-j}\over
%j!(m-j)!(n-j)!(\nu)_{m+n-j}} {(m+n-2j)!\over(2\nu)_{m+n-2j}} 
%{m+n+\nu-2j\over\nu} C_{m+n-2j}^\nu(x) $$
%$$ Q_m  Q_n = {m+\nu\over \nu} {n+\nu\over \nu} \sum_{j=0}^{\min\{m,n\}} 
%{\nu\over(m+n+\nu-j)} {(\nu)_j(\nu)_{m-j}(\nu)_{n-j}(2\nu)_{m+n-j}\over
%j!(m-j)!(n-j)!(\nu)_{m+n-j}} {(m+n-2j)!\over(2\nu)_{m+n-2j}} Q_{m+n-2j} $$
%$$ Q_m  Q_n = \sum_{j=0}^{\min\{m,n\}} { (m+\nu)(n+\nu)
% (m+n-2j)!  (\nu)_j(\nu)_{m-j}(\nu)_{n-j}(2\nu)_{m+n-j} \over \nu (m+n+\nu-j)
%j!(m-j)!(n-j)!(\nu)_{m+n-j} (2\nu)_{m+n-2j} } Q_{m+n-2j}, $$
\begin{equation}
\label{Qkprods}
Q_k^{(d)}  Q_l^{(d)} = \sum_{j=0}^{\min\{k,l\}} { (k+\nu)(l+\nu)
 (k+l-2j)!  (\nu)_j(\nu)_{k-j}(\nu)_{l-j}(2\nu)_{k+l-j} \over \nu (k+l+\nu-j)
j!(k-j)!(l-j)!(\nu)_{k+l-j} (2\nu)_{k+l-2j} } Q_{k+l-2j}^{(d)}, 
%\quad \nu:={d-2\over 2},
\end{equation}
where $\nu:={d-2\over 2}$, $(\nu)_j$ is the Pochhammer symbol,
and the coefficients are clearly positive.
%We will denote by $k\cdot l$ the degrees of the polynomials
%occuring in the Gegenbauer expansion (\ref{Qkprods}), i.e.,
We will denote by $k\cdot l$ the degrees of the Gegenbauer polynomials 
occuring in (\ref{Qkprods}), i.e.,
\begin{equation}
\label{realGegindexproduct}
k\cdot l :=\{k+l-2j:0\le j\le m \}, \qquad m:=\min\{k,l\},
\end{equation}
which we extend to subsets in the natural way
\begin{equation}
\label{realGegindexproductsubsets}
K\cdot L := \bigcup_{k\in K\atop l\in L} k\cdot l.
\end{equation}
As an example, for $L=\{1,4\}$, we have 
$L\cdot L=\{0,2,3,4,5,6,8\}$.
%, and for $d=2$ the result
%can presumably be obtained from a formal expansion (and simplification).

\begin{lemma}
\label{Qkprodlemma}
 %Let $L,K \subset \{0,1,2,\ldots\}$ be finite.  
If $F$ and $G$ give potentials for %unitarily invariant subspaces 
$P_L$ and $P_K$ (on the real sphere),
then $FG$ gives a potential for $P_{L\cdot K}$.
\end{lemma}

\begin{proof} Multiply out the Gegenbauer expansions for $F$ and $G$
	and use (\ref{Qkprods}).
\end{proof}

We now present the analogue of Lemma \ref{Qkprodlemma}
for the product of complex Gegenbauer polynomials.
It can be shown (see \cite{R80}) that
\begin{equation}
\label{H(p,q)multrule}
H(p,q) H(r,s) \subset \sum_{j=0}^\mu H(p+r-j,q+s-j), \quad
\mu:=\min\{p,s\}+\min\{q,r\}, 
\end{equation}
where there is equality for $d\ge 3$.
Motivated by this, 
%For a set of indices $\tau$, we define 
%$$ \tau^{\rm rev} := \{(q,p): (p,q)\in\tau\}. $$
for indices $(p,q)$ and $(r,s)$, we define the operation
\begin{equation}
\label{gegindexproduct}
%(p,q)*(r,s) &:= (p,q)+(s,r)=(p+s,q+r), \cr
(p,q)\cdot (r,s) := \bigcup_{j=0}^\mu \{ (p+r-j,q+s-j)\}, \quad
\mu:=\min\{p,s\}+\min\{q,r\},
\end{equation}
which we extend to subsets of indices as in (\ref{realGegindexproductsubsets}).
%$$ % \cU*\cV:= \{\ga*\gb:\ga\in\cU,\gb\in\cV\}=\cU+\cV^{\rm rev}, \qquad 
%\cU\cdot\cV:=\bigcup_{\ga\in\cU\atop\gb\in\cV} \ga\cdot\gb. $$
As examples, we have
\begin{align*}
%(p,q)*(q,p) &= (2p,2q), \cr
(p,q)\cdot(q,p) &= \{(0,0),(1,1),\ldots,(k,k)\}, \quad k=p+q,
\end{align*}
so that $(0,0)\in\cU\cdot\cU^\rev$, when $\cU$ is nonempty,
and
\begin{align*}
%(p,q)*(p,q) &= (p+q,p+q), \cr
(p,q)\cdot(p,q) &= 
\{(2p,2q),(2p-1,2q-1),\ldots,(2p-2m,2q-2m)\}, \quad m=\min\{p,q\}. 
\end{align*}
%The operation $*$ (called a convolution) was introduced by \cite{RS14}, and used to understand
%the products of polynomials in the spaces $P_L$, where $L$ is a lower set.
%A set of indices $\cU$ is said to be a {\bf lower set} if $(k,l)\in\cU$
%whenever there is a $(p,q)\in\cU$ with $k\le p$, $l\le q$. 
In view of (\ref{H(p,q)multrule}) and (\ref{gegindexproduct}), 
%, which motivates the definition of $(p,q)\cdot(r,s)$.
%From this, 
it follows that
$$ P_L P_K \subset \sum_{(p,q)\in L\cdot K} H(p,q). $$
The corresponding analogue of the Roger's formula (\ref{Qkprods}) is
$$  Q_{pq}^{(d)} Q_{rs}^{(d)} = \sum_{j=0}^\mu c_{p+r-j,q+s-j}^{(d)} 
Q_{p+r-j,q+s-j}^{(d)}, \quad
\mu:=\min\{p,s\}+\min\{q,r\}, $$
with nonnegative coefficients (see \cite{CW18}).
The nonnegativity of the coefficients follows from the 
Schur product theorem.
This gives the complex analogue of Lemma \ref{Qkprodlemma}.

\begin{theorem} 
\label{Qpqprodtheorem}
The product of functions giving a potential
for the real or complex sphere
 gives 
a potential. Indeed, if $F$ and $G$ are give potentials for 
$P_\cU$ and $P_\cV$, then $FG$ gives a potential for 
a subspace of $P_{\cU\cdot\cV}$, which is all of $P_{\cU\cdot\cV}$
for the real sphere and the complex sphere when $d\ge3$.
\end{theorem}

The following particular case % result 
will be useful.
% is useful for finding upper bounds for designs 
%from a potential $F$ with a zero constant term.

\begin{corollary}
\label{Qpqprodthcorollary}
If $F=\sum_{(p,q)} f_{pq} Q_{pq}$, $f_{pq}\ge0$ is a potential for
the complex sphere, then 
$G=\overline{Q_{ab}} F/Q_{ab}(1)$
%$G(z)=\overline{Q_{ab}(z)} F(z)/Q_{ab}(1)$
is a potential with constant term $g_0=\inpro{G,1}_\cgeg=f_{ab}$.
\end{corollary}

\begin{proof} Since $\overline{Q_{ab}}=Q_{ba}$, $G$ is a potential, 
and its constant term is
%, as follows
$$ g_0 %=\inpro{\overline{z} F(z),1}_\cgeg
= \inpro{\overline{Q_{ab}} F/Q_{ab}(1),1}_\cgeg
={1\over Q_{ab}(1)} \inpro{ \sum_{(p,q)} f_{pq} Q_{pq}, Q_{ab}}_\cgeg
={f_{ab}\over Q_{ab}(1)} \inpro{Q_{ab}, Q_{ab}}_\cgeg
%={1\over Q_{ab}(1)} f_{ab} Q_{ab}(1) 
= f_{ab}, $$
as desired.
\end{proof}

\begin{example} 
\label{conjzFresult}
Since $\overline{z}=\overline{Q_{10}(z)}/Q_{10}(1)$, 
we have that if $F$ is a potential for the complex sphere, then $G(z)=\overline{z}F(z)$ is
a potential with constant term $g_0=f_{10}$.
\end{example}

\begin{remark}
It is shown in \cite{R80} (Theorem 12.5.10), 
that for $d=2$ the only time there is not equality
in (\ref{H(p,q)multrule}) is when $(p,q)=(r,s)$, in which case 
$$ H(p,q) H(r,s) = \sum_{j=0\atop j\, {\rm even}}^\mu H(p+r-j,q+s-j), \quad
\mu  %:=\min\{p,s\}+\min\{q,r\}
=2\min\{p,q\}.  $$
If the product $\cdot$ of (\ref{gegindexproduct}) is modified in this case to
$\cdot_2$, with
$$ (p,q)\cdot_2 (p,q) 
:= \bigcup_{j=0\atop j\, {\rm even}}^\mu \{ (2p-j,2q-j)\}
\subset (p,q)\cdot(p,q), \qquad
%= \bigcup_{j=0}^\mu \{ (2p-j,q2-j)\}, \quad
\mu:=2\min\{p,q\}, $$
then the product $FG$ of Theorem \ref{Qpqprodtheorem}, is a potential 
for (all of) $P_{\cU\cdot_2\cV}$.
\end{remark}

\section{Bounds for real and complex spherical designs}
\label{boundsection}

Here we consider the relationship between our results on potentials
and the seminal paper \cite{DGS77} on codes and spherical designs.
In particular, we seek to understand the given bounds
%upper/absolute/special and lower/Fisher type bounds 
on the number of vectors in real spherical designs,
and then extend them in a natural way to complex spherical designs.

It is now convenient to allow the univariate polynomial $F$ giving a potential
for real or complex spherical designs $\Phi=(v_j)_{j=1}^n$ 
to have a (possibly nonzero) constant term $f_0$ in its Gegenbauer expansion,
so that the potential (with constant) satisfies
$$ n^2 A_F(\Phi) =\sum_j\sum_k F(\inpro{v_j,v_k})\ge n^2f_0, $$
with equality if and only if $(v_j)$ is a design.
A very specific way equality can be achieved is by having each nondiagonal term in the sum
be constant, i.e.,
\begin{equation}
\label{absolutetypeeqns}
F(\inpro{v_j,v_k})=c, \quad\forall j\ne k. 
\end{equation}
The value of the constant $c$ depends strongly on the choice of $F$, indeed
$$ n^2 A_F(\Phi)=nF(1)+(n^2-n)c=n^2 f_0 \Implies c={nf_0-F(1)\over n-1}. $$
In actuality, the number of equations in (\ref{absolutetypeeqns}) depends on
the number of {\bf angles} of $\Phi$, i.e., size of the set of inner products
$$ \Ang(\Phi):=\{\inpro{v_j,v_k}:v_j\ne v_k\}\subset [-1,1), $$
and the value of the potential depends on these angles and their multiplicities. i.e.,
\begin{equation}
\label{AFmalphaexp}
n^2 A_F(\Phi)=nF(1)+\sum_{\ga\in\Ang(\Phi)} m_\ga F(\ga) \ge n^2f_0, 
\end{equation}
where $m_\ga>0$ is the multiplicity of the angle $\ga$ as an entry
of the Gramian $[\inpro{v_j,v_k}]$.

We are now in a position to give a transparent statement and proof of
Theorem 4.3 of \cite{DGS77}. This is the key result which gives upper bounds
for the number of points in codes and designs, given a suitable choice of potential $F$.

Let $A\subset[-1,1)$. A finite set $X=(v_j)$ of $n$ unit vectors in $\Rd$ is an {\bf $A$-code} 
if its angles $\Ang(X)$ are contained in $A$.

%In these terms, the Theorem 4.3 of \cite{DGS77} is the following.

\begin{theorem} 
\label{Acodebnd} (Upper bound)
Let $F=\sum_k f_k Q_k$ be a polynomial with
%Let $F$ be a polynomial with Gegenbauer coefficients 
$f_k\ge0$, $f_0>0$, i.e., a potential with a positive constant $f_0$,
% i.e.,
%$$ A_{F,w}(\Phi):= \sum_j\sum_k w_j w_k F(\inpro{v_j,v_k}), $$
%is a potential for $P=\spam\{1\}\oplus Q$, where
%$Q=\bigoplus_{k: f_k\ne 0} \Harm_k(\Rd)$. 
%$$ P=\bigoplus_{k: f_k\ne 0} \Harm_k(\Rd). $$
%which is compatible with $A\subset[-1,1)$, i.e.,
for which 
\begin{equation}
\label{Falphacdn}
F(\ga)\le0, \qquad \ga\in A\subset[-1,1).
\end{equation}
%$F(\ga)\le0$, $\ga\in A$. 
Then the size $n$ of any $A$-code $X$ satisfies
\begin{equation}
\label{DGSineq1}
n\le {F(1)\over f_0}, 
\end{equation}
with equality if and only if
the angles of $X$ are roots of $F$, and $X$ is a spherical $P$-design for 
$$ P= P_L= \bigoplus_{k\in L} \Harm_k(\SS), \quad
L=\{k:f_k>0\}. $$
\end{theorem}

\begin{proof} Suppose that $X$ is an $A$-code. Then applying (\ref{Falphacdn}) to
(\ref{AFmalphaexp}) gives
$$ n^2f_0 \le n^2 A_F(\Phi)=nF(1)+\sum_{\ga\in\Ang(\Phi)} m_\ga F(\ga) \le nF(1), $$
which gives the inequality (\ref{DGSineq1}). There is equality above when $F(\ga)=0$, $\ga\in\Ang(X)$,
and $A_F(\Phi)=f_0$, i.e., $X$ is a spherical design for the potential $F$,
i.e., for $P=P_L$.
\end{proof}

The weighted version of (\ref{AFmalphaexp}) is
$$ A_{F}(\Phi)
 = F(1)\Bigl(\sum_{j}w_j^2\Bigr)
+ \sum_{\ga\in A} F(\ga) 
\Bigl(\sum_{j,k\atop\inpro{v_j,v_k}=\ga} w_jw_k\Bigr) \ge f_0, $$
which allows Theorem \ref{Acodebnd} to be generalised, with (\ref{DGSineq1}) becoming
$$ n \le {F(0)\over f_0} \bigl(n\sum_{j}w_j^2\bigr). $$

Before giving examples of Theorem \ref{Acodebnd}, 
we consider the corresponding lower  %lower/Fisher 
bound on $n$ given by Theorem 5.10 of \cite{DGS77},
which can be established using a similar method, where $F$ is a difference of potentials.
For $F=\sum_k f_k Q_k$ a univariate polynomial,
with Gegenbauer coefficients $f_k\in\RR$, we call
$$ F=f_0+F^+-F^-,
\qquad F^+:=\sum_{k\ne 0\atop f_k>0} f_k Q_k, \quad
F^-:= -\sum_{k\ne0\atop f_k<0} f_k Q_k, $$
its decomposition into potentials. We have the following version of (\ref{AFmalphaexp})
\begin{equation}
\label{AFmalphaexpgeneral}
\sum_j\sum_k F(\inpro{v_j,v_k}) = n^2 f_0 + n^2 A_{F^+}(\Phi)-n^2 A_{F^-}(\Phi)
= nF(1)+\sum_{\ga\in\Ang(\Phi)} m_\ga F(\ga).
\end{equation}

\begin{theorem}
\label{Rdlowerest}
(Lower bound)
Let %Suppose that 
$P=P_L$ be a unitarily invariant space of polynomials,
and $F=\sum_k f_k Q_k$ a be a univariate polynomial with
$$ \{k:f_k>0\} = L\cup\{0\}. \qquad F(\ga)\ge0, \quad \forall\ga\in[-1,1]. $$
Then the size $n$ of any $P$-design $\Phi$ satisfies
\begin{equation}
\label{DGSineq2}
n\ge {F(1)\over f_0}, 
\end{equation}
with equality if and only if
the angles of $\Phi$ are roots of $F$, and $\Phi$ is a spherical 
$P_K$-design for 
$K=\{k:f_k<0\}$.
%\begin{equation}
%\label{desartecdn2}
%F(\inpro{\xi,\eta})=0, \quad \xi\ne\eta,\ \xi,\eta\in X, \qquad 
%f_k H_k^TH_0=0, \quad k\not\in L\cup\{0\}. 
%\end{equation}
\end{theorem}

\begin{proof} By assumption, $F^+$ is a potential for $P=P_L$, and so $A_{F^+}(\Phi)=0$.
Thus (\ref{AFmalphaexpgeneral}) reduces to
$$ n^2 f_0 -n^2 A_{F^-}(\Phi) -nF(1) = \sum_{\ga\in\Ang(\Phi)} m_\ga F(\ga). $$
Since $F(\ga)\ge0$, we obtain the inequality
$$ n^2 f_0 -nF(1) \ge n^2 A_{F^-}(\Phi), $$
with equality when $F(\ga)=0$, $\ga\in\Ang(\Phi)$.
Since $F^{-}$ is a potential for $P_K$ and $f_0>0$, we have
$$ n^2 f_0 -nF(1) \ge n^2 A_{F^-}(\Phi)\ge 0
\Implies n\ge {F(1)\over f_0}, $$
which gives (\ref{DGSineq2}). Moreover, there is equality above when 
$A_{F^-}(\Phi)=0$, i.e., when $\Phi$ is a spherical $P_K$-design.  
\end{proof}

The original statement of Theorem \ref{Rdlowerest} in \cite{DGS77} (Theorem 5.10) was
for spherical $t$-designs, i.e., $L=\{0,1,2,\ldots,t\}$
(see Example \ref{t-designF}).
The weighted version of this result % Theorem \ref{Rdlowerest} 
can be obtained, in the obvious way,
giving the lower bound
$$ n \ge {F(0)\over f_0} \bigl(n\sum_{j}w_j^2\bigr). $$

The easiest way to find an $F$ satisfying $F(x)\ge0$ on $[-1,1]$, which 
we will refer to as a ``nonnegative potential'', is to take the
square of an appropriate univariate polynomial.

\begin{corollary}
\label{RdlowerestCor1}
Let $E\subset\NN$ be a nonempty finite set of indices, % containing $0$, 
and $L=E\cdot E$.
Then 
\begin{equation}
\label{RdlowerestCor1F}
F:= \Bigl(\sum_{k\in E} Q_k\Bigr)^2,
\end{equation}
gives a nonnegative potential for $P_L$, and the number of points $n$ in a 
$P_L$-design satisfies
\begin{equation}
\label{RdlowerestCor1nest}
n\ge {F(1)\over f_0}
= \sum_{k\in E} Q_k(1) = \sum_{k\in E}\dim(\Harm_k(\Rd)). 
\end{equation}
\end{corollary}

\begin{proof} Clearly, $F$ is nonnegative, and by Lemma \ref{Qkprodlemma}, it is
a potential for $L=E\cdot E$. Thus, we may apply Theorem \ref{Rdlowerest}.
By the orthogonality relations (\ref{realGegorthog}), we have
$$ f_0 =\inpro{F,1}_\geg
= \inpro{\Bigl(\sum_{k\in E} Q_k\Bigr)^2,1}_\geg
= \inpro{\sum_{k\in E} Q_k,\sum_{\ell\in E} Q_\ell}_\geg
= \sum_{k\in E} \inpro{Q_k,Q_k}_\geg
= \sum_{k\in E} Q_k(1), $$
so that 
$$ {F(1)\over f_0}
= \sum_{k\in E} Q_k(1) = \sum_{k\in E}\dim(\Harm_k(\Rd)), $$
and we obtain the desired estimate.
\end{proof}

\begin{corollary}
\label{RdlowerestCor2}
Let $E\subset\NN$ be a finite set of even indices or of odd indices, and 
\begin{equation}
\label{RdlowerestCor2F}
F:= \Bigl({Q_1\over d}+1\Bigr) \Bigl(\sum_{k\in E} Q_k\Bigr)^2, \qquad
L=\{0,1\}\cdot(E\cdot E).
\end{equation}
Then $F$ gives a nonnegative potential for $P_L$, and the number of points $n$ in a
$P_L$-design satisfies
\begin{equation}
\label{RdlowerestCor2nest}
n\ge {F(1)\over f_0}
= 2 \sum_{k\in E} Q_k(1) = 2 \sum_{k\in E}\dim(\Harm_k(\Rd)).
\end{equation}
\end{corollary}

\begin{proof} The proof is similar to that of Corollary \ref{RdlowerestCor1}.
The first factor of $F$ satisfies % since the factor ${Q_1(x)\over d}+1=x+1\ge0$, $x\in[-1,1]$,
$$ {Q_1(x)\over d}+1=x+1\ge0, \qquad x\in[-1,1], $$
and so $F$ gives a nonnegative potential for $L$.
%, and we may apply Theorem \ref{Rdlowerest}.
The polynomial $Q_1(\sum_k Q_k)^2$ is odd, so has zero integral with
respect to the Gegenbauer weight, and we have
$$ f_0 
=\inpro{F,1}_\geg
= \inpro{\Bigl({Q_1\over d}+1\Bigr)\Bigl(\sum_{k\in E} Q_k\Bigr)^2,1}_\geg
= \inpro{\Bigl(\sum_{k\in E} Q_k\Bigr)^2,1}_\geg
= \sum_{k\in E} Q_k(1)>0. $$
Further,
$$ F(1)= 2 \Bigl(\sum_{k\in E} Q_k(1)\Bigr)^2, $$
and so we obtain the desired estimate from Theorem \ref{Rdlowerest}.
\end{proof}

We observe that the choice of the polynomial $\sum_{k\in E} Q_k$ in 
Corollaries \ref{RdlowerestCor1} and \ref{RdlowerestCor2} is optimal. 
Indeed, if a different convex combination
$\sum_{k}c_k Q_k$ is taken in the potential $F$, 
then term 
$\sum_k Q_k(1)$ in the lower estimate
becomes
$$ M = { \Bigr(\sum_{k}c_k Q_k(1)\Bigr)^2 \over \sum_{k}c_k^2 Q_k(1) }. $$
By Cauchy-Schwarz, we have 
$$ \Bigl(\sum_k c_k Q_k(1)\Bigr)^2
= \Bigl(\sum_k c_k\sqrt{Q_k(1)}\sqrt{Q_k(1)}\Bigr)^2
\le\Bigl( \sum_k c_k^2Q_k(1)\Bigr) \Bigl( \sum_k Q_k(1)\Bigr), $$
so that 
$$ M\le \sum_k Q_k(1), $$
with equality if and only if $c_k=1$, $\forall k$.

Spherical designs which give equality in one of the bounds of
Theorems \ref{Acodebnd} and \ref{Rdlowerest}, are said to be {\bf tight} (not to be confused with tight frames).
Such designs are very special, and have played a prominent role in the theory of spherical 
designs: since the angles of tight designs are roots of $F$, it is possible to investigate
their existence. 

This following example is Theorems 5.11 and 5.12 of \cite{DGS77}.

\begin{example} 
\label{t-designF} (Spherical $t$-designs) These are given by $L=\{0,1,2,\ldots,t\}$, which
can be obtained by the following choices (for $t$ even and odd). 
\begin{align*}
 & t=2e:   \quad\qquad L=E\cdot E, \quad E=\{0,1,2,\ldots,e\}, \cr
 & t=2e+1:   \qquad L=\{0,1\}\cdot(E\cdot E), \quad E=\{e,e-2,e-4,\ldots\}.
\end{align*}
The corresponding estimates (\ref{RdlowerestCor1nest}) and (\ref{RdlowerestCor2nest}) are
$$ n\ge \sum_{k=0}^e\dim\bigl(\Harm_k(\Rd)\bigr) 
= \dim(\Pi_e(\Rd)|_\SS) = {e+d-1\choose d-1}+{e+d-2\choose d-1}, \quad t=2e, $$
$$ n \ge 2 \sum_{0\le j\le e/2} \dim\bigl(\Harm_{e-2j}(\Rd)\bigr)
= 2\dim(\Hom_e(\Rd)|_\SS) 
= 2{ e+d-1\choose d-1}, \quad t=2e+1. $$
In \cite{BRV13}, it is shown that there exist spherical $t$-designs
whose number of points has this order of growth in $t$, i.e., 
with $n\ge c_d t^{d-1}$.
\end{example}

\begin{example} (Spherical half-designs) The spherical half-designs of even
order $m=2t$ are given by 
$$ L = \{2t,2t-2,2t-4,\ldots\}=E\cdot E, \quad E=\{t,t-2,t-4,\ldots\}. $$
By applying Corollary \ref{RdlowerestCor1}, we obtain the estimate
$$ n \ge \sum_{0\le j\le t/2} \dim\bigl(\Harm_{t-2j}(\Rd)\bigr)
= \dim(\Hom_t(\Rd)|_\SS) 
= { t+d-1\choose d-1}. $$
In \cite{DGS77}, it is shown that for $t=2m+1$ odd, a tight spherical $t$-design 
consists the vectors of a tight spherical-half design of order $2m$, and its negatives.
\end{example}

%That being said, the art is in choosing a suitable $F$, which we were
%unable to do for the half-designs of odd order. 
%As an example, 
%$$ F(x)= ({Q_1\over d}+1)^t=(1+x)^t $$
%gives a nonnegative potential for the spherical $t$-designs, but a very poor estimate.
%Its roots are $x=-1$, which can only be an angle for a vector and its negative.
A second natural way to try and find suitable potentials, 
is to optimise over all possible such potentials. This is the ``linear programming method''
of \cite{DGS77}, e.g., from Theorem \ref{Rdlowerest}, 
%, more recently, ``semidefinite programming'' method, e.g., 
one has the lower bound for $P_L$-designs
$$ n \ge \max\Bigl\{ {F(1)\over f_0}: F=\sum_k f_k Q_k, \{k:f_k>0\}=L\cup\{0\}, 
\hbox{$F\ge0$ on $[-1,1]$}\Bigr\},  $$
and the following example.

\begin{example} (Spherical designs of harmonic index $t$) 
A polynomial $F$ giving potential for the spherical designs for 
$P_L=\Harm_t(\Rd)$ (and no larger space), 
to which we can apply Theorem \ref{Rdlowerest}, 
has the form $F=Q_t+c$, for $c\ge b:=-\min_{x\in[-1,1]}Q_k(x)>0$.
The corresponding estimate $n\ge F(1)/f_0 = Q_t(1)/c+1$ is optimised
by taking $c=b$. This estimate, and variants of it, can be found in
\cite{ZBBJY17}.
\end{example}

The linear programming method has recently been applied to real spherical $(t,t)$-designs
\cite{B20},
\cite{BBDHSSb}, \cite{BBDHSS25a} (weighted designs).

We now consider bounds for complex designs.
As before, we will take % consider 
polynomials
of the form 
$$ F=\sum_{(p,q)} f_{pq} Q_{pq}, \qquad f_{pq}\in\RR. $$
%$F=\sum_{(p,q)} f_{pq} Q_{pq}$, where $f_{pq}\in\RR$.
These are polynomials of a complex variable with real coefficients,
and so take complex values in general (unless $f_{pq}=f_{qp}$, $\forall (p,q)$). % (despite having real coefficients). 
Nevertheless, they do have $F(1),f_0\in\RR$,
and, most importantly,
\begin{equation}
\label{complexpotentialrealexp}
\sum_{\ga\in A} F(\ga) m_\ga = \sum_{\ga\in A} \Re(F(\ga)) m_\ga.
\end{equation}
The last equation follows since $\overline{F(z)}=F(\overline{z})$,
and an angle $\ga$ and % its conjugate
$\overline{\ga}$ appear with the same multiplicity $m_\ga=m_{\overline{\ga}}$
in (\ref{AFmalphaexp}), 
so the sum of the corresponding pair of 
terms is
\begin{align*}
m_\ga F(\ga) +  m_{\overline{\ga}} F(\overline{\ga})
&= m_\ga F(\ga) + m_\ga \overline{F(\ga)}
= 2 m_\ga \Re(F(\ga)) \cr
&= m_\ga \Re(F(\ga)) + m_{\overline{\ga}} \Re(\overline{F(\ga)})
= m_\ga \Re(F(\ga)) + m_{\overline{\ga}} \Re(F(\overline{\ga})). 
\end{align*}

In view of (\ref{complexpotentialrealexp}), the extension of 
Theorems \ref{Acodebnd} and \ref{Rdlowerest} to the complex case become obvious,
and we combine them.

%Moreover, since $\overline{Q_{pq}}=Q_{qp}$,
%they can be made to be real valued by choosing $f_{pq}=f_{qp}$. 

%The main subtlety in extending Theorems \ref{Acodebnd} and \ref{Rdlowerest}
%is the estimate of 
%$$ \sum_{\ga\in A} F(\ga) m_\ga = \sum_{\ga\in A} \Re(F(\ga)) m_\ga, $$
%and then of course finding suitable functions $F$. 

%For Theorem \ref{Acodebnd}, one has $A\subset\DD\setminus\{-1\}$,
%and the condition of $F(\ga)$ can be taken to be
%$F(\ga)\le0$, $\ga\in A$ or $\Re(F(\ga))\le0$, $\ga\in A$, with 
%the corresponding equality conditions $F(\inpro{\xi,\eta})=0$, $\xi\ne\eta$
%or $\Re(F(\inpro{\xi,\eta}))=0$, $\xi\ne\eta$, respectively.
%Similarly, for Theorem \ref{Rdlowerest}, the condition $F(\ga)\ge0$, $\ga\in[-1,1]$ can be replaced by either 
%$F(\ga)\ge0$, $\ga\in\DD$ or $\Re(F(\ga))\ge0$, $\ga\in\DD$, with the
%same change to the conditions for equality as above. 

\begin{theorem} 
\label{complexupperlowerbnds}
(Upper and Lower bounds)
Let $F=\sum_{(p,q)} f_{pq} Q_{pq}$ be a polynomial with 
$$ f_{pq}\in\RR, \qquad f_0=f_{00}>0, \qquad \tau=\tau^+:=\{(p,q):f_{pq}>0\},
\quad \tau^-:=\{(p,q):f_{pq}<0\}, $$
and $A\subset\{z\in\CC:|z|\le1,z\ne1\}$. %  \DD\setminus\{1\}=
Then
\begin{enumerate}[(i)]
\item If $F$ is a potential, i.e., $\tau^-=\{\}$, and % with a positive constant $f_0$, 
$\Re(F(\ga))\le0$, $\forall\ga\in A$, then the size $n$ of any $A$-code $X$ satisfies
$$ n\le {F(1)\over f_{0}}, $$
with equality if and only if the angles of $X$ are roots of $F$, and $X$ is a %spherical
$\tau$-design 
%$$ P=P_\tau=\bigoplus_{(p,q)\in\tau} H(p,q), \qquad \tau=\{k:f_k>0\}. $$
\item If $\Re(F(\ga))\ge0$, $\ga\in\{z\in\CC:|z|\le 1\}$, then
the size $n$ of any spherical $\tau$-design $X$ satisfies
$$ n\ge {F(1)\over f_{0}}, $$
with equality if and only if the angles of $X$ are roots of $F$, 
and $X$ is a $\tau^-$-design.
\end{enumerate}
\end{theorem}

This result can be found in \cite{RS14}, but with $F(\ga)$ in place of
our $\Re(F(\ga))$, as the subtlety that $F(\ga)$ can be complex was not considered.
For a finite set of indices $\cE$, 
\begin{equation}
\label{QEdefn}
Q_\cE:=\sum_{(p,q)\in\cE} Q_{pq}.
\end{equation}
%so that
%$$ Q_\cE(1)=\sum_{(p,q)\in\cE} Q_{pq}(1)
%=\sum_{(p,q)\in\cE}\dim\bigl(H(p,q)\bigr)=\dim(P_\cE). $$
As for real designs, we will say that a complex spherical design $X$ is {\bf tight} 
if it meets one of the bounds of Theorem \ref{complexupperlowerbnds}, 
i.e., 
$$n=|X|={F(1)\over f_0}, \qquad \hbox{the angles of $X$ are roots of $F$}. $$

\begin{corollary}
\label{CdlowerestCor1}
Let $\cE\subset\NN^2$ be a nonempty finite set of indices. % containing $0$.  
%, and $L=E\cdot E$.
Then
\begin{equation}
\label{CdlowerestCor1F}
F:= 
|Q_\cE|^2
% \Bigl\vert\sum_{(p,q)\in\cE} Q_{pq}\Bigr\vert^2
= \Bigl(\sum_{(p,q)\in\cE} Q_{pq}\Bigr) \Bigl(\sum_{(r,s)\in \cE^\rev} Q_{rs}\Bigr)\ge0,
\end{equation}
is a %nonnegative 
potential for $P_{\cE\cdot\cE^\rev}$, and the number of points $n$ in a
$(\cE\cdot\cE^\rev)$-design satisfies 
\begin{equation}
\label{CdlowerestCor1nest}
n \ge    % {F(1)\over f_0}
\dim(P_\cE) 
= \sum_{(p,q)\in\cE} \dim\bigl(H(p,q)\bigr),
%= \sum_{(p,q)\in\cU} Q_{pq}(1), 
\end{equation}
with equality if and only if the angles of the design are roots of the polynomial
$Q_\cE$.
%$\sum_{(p,q)\in\cE} Q_{pq}$.
\end{corollary}

%\begin{theorem}
%\label{tightfisher}
%(Tight complex spherical designs)
%Let $\cU$ be a finite set of indices. % which contains $(0,0)$.
%Then the number of points $n$ in a spherical
%$(\cU\cdot\cU^\rev)$-design satisfies
%\begin{equation}
%\label{tightfbound}
%n \ge \dim(P_\cU) = \sum_{(p,q)\in\cU} Q_{pq}(1),
%\end{equation}
%with equality if and only if the angles are roots of the polynomial
%$f_\cU:=\sum_{(p,q)\in\cU} Q_{pq}$.
%\end{theorem}

\begin{proof}
Since $\overline{Q_{pq}}=Q_{qp}$, we have the equality in (\ref{CdlowerestCor1F}),
and the multiplication rule of Theorem \ref{Qpqprodtheorem} implies that
$F=Q_\cE\overline{Q_\cE}$ is a potential for $\cE\cdot\cE^\rev$,
with
%$$ F:=\Bigl\vert\sum_{(p,q)\in\cU} Q_{pq}\Bigr\vert^2
%=\Bigl(\sum_{(p,q)\in\cU} Q_{pq}\Bigr) \Bigl(\sum_{(r,s)\in\cU^\rev} Q_{rs}\Bigr)
%=\sum_{(p,q)\in\cU\cdot\cU^\rev} f_{pq}Q_{pq} $$
%defines a potential for $\cU\cdot\cU^\rev$, with
$$ f_{00} 
%= \inpro{\Bigr\vert\sum_{(p,q)\in\cU} Q_{pq}\Bigr\vert^2,1}_\cgeg
= \inpro{F,1}_\cgeg
=\inpro{ \sum_{(p,q)\in\cE} Q_{pq}, \sum_{(r,s)\in\cE} Q_{rs} }_\cgeg
= \sum_{(p,q)\in\cE} \inpro{Q_{pq},Q_{pq}}_\cgeg
= \sum_{(p,q)\in\cE} Q_{pq}(1).  $$
Since $F=|Q_\cE|^2\ge0$, by construction, Theorem \ref{complexupperlowerbnds} gives the estimate
$$ n \ge {F(1)\over f_{00}}
%= {\Bigl\vert\sum_{(p,q)\in\cE} Q_{pq}(1)\Bigr\vert^2\over\sum_{(p,q)\in\cE} Q_{pq}(1)}
= {|Q_\cE(1)|^2\over Q_\cE(1) }
= Q_\cE(1) 
= \sum_{(p,q)\in\cE} Q_{pq}(1) 
=\sum_{(p,q)\in\cE} \dim\bigl(H(p,q)\bigr) = \dim(P_\cE), $$
with equality if and only if the angles are roots of 
$Q_\cE$.  %= \sum_{(p,q)\in\cE} Q_{pq}. $
\end{proof}

The inequality (\ref{CdlowerestCor1nest}) is given in \cite{RS14} (Theorem 4.2) for $\cU$ a
lower set, via a different argument. We now explain, and give the argument, which is classical 
and neat. A ``convolution'' product $*$ on indices (and sets of indices) is given in \cite{RS14} by
$$ \cE * \cE := \cE+\cE^\rev\subset \cE\cdot\cE^\rev. $$
For $\cE$ a lower set, one has $\cE * \cE=\cE\cdot\cE^\rev$, and so Theorem \ref{CdlowerestCor1}
can be applied.
Let $X$ be a spherical $(\cE\cdot\cE^\rev)$-design, and 
$(f_j)$ be an orthonormal basis for
$P_\cE=\oplus_{(p,q)\in\cE} H(p,q)$. 
It follows from (\ref{H(p,q)multrule}) that
$$ f_j \overline{f_k} \in 
\Bigl( \bigoplus_{(p,q)\in\cE } H(p,q) \Bigr)
\Bigl( \bigoplus_{(r,s)\in\cE^\rev} H(r,s)  \Bigr)
\subset \bigoplus_{(p,q)\in\cE\cdot\cE^\rev} H(p,q), $$
and so we have, by the cubature rule, that
$$ {1\over|X|} \sum_{x\in X} f_j(x)\overline{f_k(x)}
=\int_\SS f_j\overline{f_k}\,d\gs =\gd_{jk}, $$
i.e., $(f_j|_X)$ is orthonormal in $\CC^X$,
and hence $n=|X|\ge\dim(P_\cE)$.

We observe that the sets 
$\cU=\cE\cdot\cE^\rev=\cE^\rev\cdot\cE=\cE^\rev\cdot(\cE^\rev)^\rev$ 
in Corollary \ref{CdlowerestCor1}
are ``symmetric'' in the sense that $\cU=\cU^\rev$, i.e., 
$$ (p,q)\in \cE\cdot\cE^\rev \Iff (q,p)\in \cE\cdot\cE^\rev. $$
The indices $(p,p)$, which give the projective designs, 
are called {\bf projective indices}.
Here is an example where $\cE$ is not a lower set.

\begin{example} (Projective indices) For any $\cE$ of the form
\begin{equation}
\label{cEpqm}
\cE=\{(p,q),(p-1,q-1),\ldots,(p-m,q-m)\}, \qquad 0\le m\le\min\{p,q\},
\end{equation}
we have
$$ \cE\cdot\cE^\rev = \{(0,0),(1,1),\ldots,(t,t)\}, \qquad t=p+q, $$
a sequence of consecutive projective indices.
Hence for any $\cU$,  we have
$$ \{(0,0),(1,1),\ldots,(t,t)\} \subset \cU\cdot\cU^\rev, 
\qquad t:=\max_{(p,q)\in\cU}(p+q).
$$
Since we can take $\cE=\{(p,q)\}$, a single point, we observe that
$\cU\cdot\cU^\rev$ cannot be a ``small'' set.
\end{example}

\begin{example}
\label{ttdesignlowerbnd}
(Spherical $(t,t)$-designs)
Applying Corollary \ref{CdlowerestCor1} for the set $\cE$ of (\ref{cEpqm}) 
for $p+q=t$ and $m=\min\{p,q\}$, gives 
the estimate
$$ n \ge \sum_{j=0}^{\min\{p,q\}} \dim\bigl(H(p-j,q-j)\bigr)
= \dim\bigl(\Hom(p,q)\bigr) = {p+d-1\choose d-1}{q+d-1\choose d-1}, $$
for $n$ the number of points in a spherical $(t,t)$-design for $\Cd$. 
The best estimate from those above is obtained for the choice $p=\lfloor {t\over2}\rfloor$, 
which gives
\begin{equation}
\label{new(t,t)-designlowerbound}
n \ge 
\begin{cases}
{k+d-1\choose d-1}^2, & t=2k; \cr
{k+d-1\choose d-1}{k+d\choose d-1}, & t=2k+1.
\end{cases}
\end{equation}
This estimate improves that given in \cite{W18} (Exercise 6.22), \cite{B20}, i.e.,
$$ n\ge{t+d-1\choose d-1}. $$
The bound, and %corresponding 
function $Q_\cE$ with $F=|Q_\cE|^2$, for the first
few values of $t$ are (respectively)
\begin{align*} t=1:  & \qquad n\ge d, \qquad\qquad\qquad Q_{10}(z)=dz, \cr
t=2: & \qquad n\ge d^2, \ \, \quad\qquad\qquad Q_{11}(z)+Q_{00}(z)=d\bigl((d+1)|z|^2-1\bigr), \cr
t=3: & \qquad n\ge {1\over2} d^2(d+1), \qquad  
Q_{21}(z)+Q_{10}(z)={1\over2} d(d+1) z\bigl( (d+2)|z|^2-2\bigr).
\end{align*}
The reverse inequality appears in \cite{H82} (Theorem 3.2) as an absolute bound.
The roots of the above polynomials give the absolute value $|z|$ of the angles 
for ``tight'' designs. The possible values for $|z|^2$ in a tight design for
$t=1,2,\ldots,5$ are
$$ \{0\},\quad \{{1\over d+1}\}, \quad \{0,{2\over d+2}\}, \quad 
\{{\sqrt{2(d+1)/(d+2)}\pm2\over d+3}\}, \quad
\{0,{\sqrt{3(d+1)/(d+3)}\pm3\over d+4}\}. \quad
$$
Tight designs exist for $t=1$ (orthonormal bases)
and $t=2$ (complex equiangular lines). %, various values of $d$). 
There are also known examples for $t=3$ of $6$ lines at ``angles'' $|z|^2=0,{1\over2}$ 
in $\CC^2$, $40$ lines at angles $0,{1\over3}$ in $\CC^4$ 
and $126$ lines at angles $0,{1\over4}$ in $\CC^6$ 
(see \cite{H82}, \cite{HW18}). 
For $t=5$, an example of $12$ lines in $\CC^2$ was given in \cite{HW18},
with angles ${1\over2}({1\over\sqrt{5}}\pm1)$.
%$$ t=4, \quad n\ge {d^2(d+1)^2\over 4}, \quad
%Q_{22}(z)+Q_{11}(z)+Q_{00}(z) = {d(d+1)\over4} \bigl( (d+2)(d+3)|z|^4-4(d+2)|z|^2+2\bigr). $$
\end{example}

Here is an example where $\cE$ is a lower set.

%\begin{example} (Complex tight frames) 
%Let $\cU=\{(1,0)\}$ (or $\cU=\{(0,1)\}$). 
%Then 
%$$ \cU\cdot\cU^\rev=\{(0,0),(1,1)\}, \quad 
%f_\cU(z)= Q_{10}(z)=dz, \qquad  n\ge d. $$
%The corresponding designs are spherical $(1,1)$-designs 
%(aka unit-norm tight frames). The {\it tight} ``unit norm tight frames''
%have angles equal to $0$ (the roots of $f_\cU$), and exist and 
%are given by the orthonormal bases.
%\end{example}

\begin{example} (The simplex)
Let $\cE=\{(0,0),(1,0)\}$, which is % (or $\cU^\rev$), %which is 
a lower set. Then
$$ \tau=\cE\cdot\cE^\rev =\{(0,0),(0,1),(1,0),(1,1)\}, $$
and the bound and polynomial $Q_\cE$ for the class of $\tau$-designs 
given by Corollary \ref{CdlowerestCor1} is
$$ n\ge d+1, \qquad Q_\cE(z)=Q_{00}(z)+Q_{10}(z)=dz+1. $$
Here the  $\tau$-designs are the balanced unit norm tight frames
% which are balanced
(see Examples \ref{Balancedsetexample} and \ref{complextightframeex}).
%also integrate the homogeneous linear polynomials (and so the vectors add to zero). 
For such a design to be tight, it must have $d+1$ vectors with angles ${-1\over d}$. 
There is a unique such configuration, given by the $d+1$
vertices of the regular simplex in $\Rd\subset\Cd$.
\end{example}

The polynomials of Example \ref{ttdesignlowerbnd}, i.e., $Q_{\cE_t}(z)$ where
$$ 
\cE_t  
%=\gep= \{(\lfloor{t\over2}\rfloor,t-\lfloor{t\over2}\rfloor), (\lfloor{t\over2}\rfloor-1,t-\lfloor{t\over2}\rfloor-1),\ldots\}
%=\{(k+\gep,k),(k+\gep-1,k-1),\ldots\},
=\{(k+\gep-j,k-j)\}_{0\le j\le k}, 
\qquad t=2k+\gep, \quad \gep=0,1, %=2t-\lfloor{t\over2}\rfloor=0,1, 
$$
appear, 
%The polynomials of Example \ref{ttdesignlowerbnd} appear, 
implicitly, in
\cite{H82}. We now explain this connection, and show that they are Jacobi polynomials,
which has implications for the location and spacing of their roots.
Hoggar defines polynomials of degree $k$ by
\begin{equation}
\label{Qkgepdefn}
Q_k^\gep(x) := {(md)_{2k+\gep}\over(m)_{k+\gep}k!}\sum_{i=0}^k(-1)^i{k\choose i}
{{}_i(k+m+\gep-1)\over{}_i(2k+md+\gep-2)}x^{k-i}, \qquad \gep=0,1,
\end{equation}
$$ {}_i(x)=x(x-1)\cdots(x-i+1), \qquad (x)_i=x(x+1)\cdots(x+i-1), $$
which depend on $\gep=0,1$, and a parameter $m$, with $m={1\over2}$ being the real case, 
and $m=1$ the complex case. 
%Here ${}_i(x)=x(x-1)\cdots(x-i+1)$, $(x)_i=x(x+1)\cdots(x+i-1)$.
It is easily verified that these are related to our Gegenbauer polynomials as follows
%$$ Q_{2k+\gep}(x) = x^\gep Q_k^\gep(x^2), \quad m={1\over2}, \qquad
%Q_{k+\gep,k}(z)=z^\gep Q_k^\gep(|z|^2), \quad m=1. $$
$$ Q_{2k+\gep}(x) = x^\gep Q_k^\gep(x^2), \qquad m={1\over2}, $$
$$ Q_{k+\gep,k}(z)=z^\gep Q_k^\gep(|z|^2), \qquad m=1. $$
It follows from the orthogonality relations for the Gegenbauer polynomials that
the $Q_k^\gep$ are orthogonal polynomials of degree $k$ on $[0,1]$, for a Jacobi weight (depending 
on $\gep$).
%, and so can be expressed in terms of Jacobi polynomials. Indeed
%$$ Q_{2k+\gep}(x) %= x^\gep Q_k^\gep(x^2)
%= \dim\bigl(\Harm_{2k+\gep}(\Rd)\bigr) x^\gep {P_k^{({d-3\over2},\gep-{1\over2})}(2x^2-1) 
%\over P_k^{({d-3\over2},\gep-{1\over2})}(1)}, $$
%$$Q_{k+\gep,k}^{(d)}(z) = \dim\bigl(H(k+\gep,k)\bigr) \,
%z^\gep {P^{(d-2,\gep)}_k (2|z|^2-1) \over P^{(d-2,\gep)}_k (1)} .  $$
%An explicit 
%formula for $R_k^\gep:=Q_0^\gep+\cdots+Q_k^\gep$ is given by \cite{H82}, and this
%has the same structural form as that for $Q_k^\gep$, and so we can deduce that
%the corresponding sums of Gegenbauer polynomials may expressed in terms 
%of Jacobi polynomials, as follows
The polynomial $Q_{\cE_t}(z)=Q_{\cE_{2k+\gep}}(z)$ appears in \cite{H82} as $z^\gep R_k^\gep(|z|^2)$,
and, by Lemma \ref{complexGegsumslemma}, and (\ref{Qpqdef}), it can be expressed as
\begin{align}
\label{Homkekequation}
Q_{\cE_{2k+\gep}}^{(d)}(z) 
&= \sum_{j=0}^k Q_{j+\gep,j}^{(d)}(z) \cr
&= {d\over 2k+\gep+d} Q_{k+\gep,k}^{(d+1)}(z) \cr
&=\dim\bigl(\Hom(k+\gep,k)\bigr) 
z^\gep {P_k^{(d-1,\gep)}(2|z|^2-1) \over P_k^{(d-1,\gep)}(1) } \cr
&= { 1 \over(d-1)! } \sum_{j=0}^k (-1)^j { (d+2k+\gep-j-1)!  \over j!(k-j)! (k+\gep-j)!}
z^\gep (|z|^2)^{k-j}.
\end{align}

We now give a complex analogue of Corollary \ref{RdlowerestCor2}. Let
$$ \cS_1 := \{(0,0),(1,0),(0,1)\}.  $$

\begin{corollary}
Let $\cE$ be a finite set of indices, with the property that
$$ (p\pm1,q),(p,q\pm1)\not\in\cE, \quad\forall (p,q)\in\cE, $$
and $\tau = \cS_1 \cdot(\cE\cdot\cE^\rev)$.
% with 
%the property that 
%$(p+1,q),(p,q+1)\not\in\cU$, $\forall (p,q)\in\cU$.
%$$ (p+1,q),(p,q+1)\not\in\cU, \quad\forall (p,q)\in\cU. $$
Then the number of points $n$ in a spherical $\tau$-design satisfies
%for $\tau = \bigcup_{|\ga|\le 1} \cU\cdot\cU^\rev+\ga$
%$P$-design for $P=P_L$, $L=\cup_{|\ga|\le 1} \cU\cdot\cU^\rev+\ga$, 
$$ n \ge 2\dim(P_\cE) = 2\sum_{(p,q)\in\cE} \dim\bigl(H(p,q)\bigr), $$
with equality if and only if the angles are $-1$ or the roots of
$Q_\cE=\sum_{(p,q)\in\cE} Q_{pq}$.
\end{corollary}

\begin{proof}
Since $z={1\over d}Q_{10}(z)$ and $\overline{z}={1\over d}Q_{01}(z)$, 
it follows from %the multiplication rule of 
Theorem \ref{Qpqprodtheorem} that
a potential for the $\tau$-designs is given by
$$ F(z):=\Bigl({z+\overline{z}\over2}+1\Bigr)
%\Bigl\vert\sum_{(p,q)\in\cU} Q_{pq}(z)\Bigr\vert^2
|Q_\cE|^2 \ge0. $$
Let $(p,q)\in\cE$. Since 
$(1,0)\cdot(p,q)=\{(p+1,q),(p,q-1)\}$, (\ref{H(p,q)multrule}) gives
$$ zQ_{pq}(z)={1\over d}Q_{10}(z)Q_{pq}(z) \in H(p+1,q)\oplus H(p,q-1) \perp P_\cE, $$
i.e., $zQ_{pq}(z)$ is orthogonal to $P_\cE$, and similarly for $\overline{z}Q_{pq}(z)$,
so that
$$ f_{00} = \inpro{F,1}_\cgeg
=\inpro{ \Bigl({z+\overline{z}\over2}+1\Bigr)
Q_\cE,Q_\cE }_\cgeg
= \sum_{(p,q)\in\cE} \inpro{Q_{pq},Q_{pq}}_\cgeg
= \sum_{(p,q)\in\cE} Q_{pq}(1).  $$
Thus, by Theorem \ref{complexupperlowerbnds}, we have the estimate
$$ n \ge {F(1)\over f_{00}}
= {2 |Q_\cE(1)|^2\over Q_\cE(1)}
= 2 Q_\cE(1) = 2\dim(P_\cU), $$
with the equality as stated.
\end{proof}

\section{Absolute and special bounds for complex designs}
\label{morecomplexbounds}

We now give an important variant of the upper bound in 
Theorem \ref{complexupperlowerbnds}, where
\begin{itemize}
\item The condition that $F$ be a potential, or even have real Gegenbauer
coefficients, is weakened.
\item The condition $F(\ga)\le0$, $\forall \ga\in A$ is strengthened to 
$F(\ga)=0$, $\forall \ga\in A$.
\item The set $A$ and $F$ depend on each other.
\item The upper bound only depends on which Gegenbauer coefficients of $F$
are nonzero.
\end{itemize}

\begin{theorem}
\label{complexfisherbounds}
Let $X=(v_j)$ be a sequence of $n$ unit vectors in $\Cd$, 
and 
$$ F=\sum_{(p,q)} f_{pq}Q_{pq}, \qquad f_{pq}\in\CC, $$ 
be a polynomial with
$$ F(1)=1, \qquad
A:=\{z\in\CC:|z|\le1,F(z)=0\}, \qquad
\tau:=\{(p,q):f_{pq}\ne0\}. $$  % =\supp(F)
Then  %the cardinality $n$ of $X$ satisfies
the $n\times n$ matrix $M=[F(\inpro{v_j,v_k})]_{j,k=1}^n$ satisfies
\begin{equation}
\label{rankMcdn}
\rank(M) \le \min\Bigl\{n,\sum_{(p,q)\in\tau}\dim(H(p,q))\Bigr\}.
\end{equation}
%$$ \min\Bigl\{n,\sum_{(p,q)\in\tau}\dim(H(p,q))\Bigr\} 
%\ge \rank\bigl([F(\inpro{v_j,v_k})]_{j,k=1}^n\bigr). $$
%$$ \min\Bigl\{n,\sum_{(p,q)\in L}\dim(H(p,q))\Bigr\} 
%\ge \rank\bigl([F(\inpro{\xi,\eta})]_{\xi,\eta\in X}\bigr). $$
In particular, if $X$ is an $A$-code, i.e., $F(\inpro{v_j,v_k})=0$, $j\ne k$,
 then 
\begin{equation}
\label{absbndeqn}
n \le \sum_{(p,q)\in\tau}\dim(H(p,q)),
\end{equation}
%$$ n \le \dim(P_L)=\sum_{(p,q)\in L}\dim(H(p,q)) = \sum_{(p,q)\in L} Q_{pq}(1), $$
with equality if and only if
\begin{equation}
\label{Fequalform}
F={1\over n} \sum_{(p,q)\in\tau} Q_{pq},
\end{equation}
in which case the angles of $X$ are roots of $F$,
and
$X$ is a spherical $\tau$-design if and only if $(0,0)\in\tau$.
%It is not clear to me what sort of spherical design it is.
\end{theorem}

\begin{proof}
Let $H_{pq}=(W_j^{(p,q)})$ be a row vector whose entries are an orthonormal basis
for $H(p,q)$, then the reproducing kernel for $H(p,q)$ is
$$ K_{pq}(\xi,\eta)=\sum_j W_j^{(p,q)}(\xi) \overline{W_j^{(p,q)}(\eta)}
=H_{pq}(\xi) H_{pq}(\eta)^* = Q_{pq}(\inpro{\xi,\eta}). $$
Using this, we may write $M$ as a product of matrices
$$ M=H_X \diag(f_{pq}I_{\dim(H(p,q)})_{(p,q)\in\tau} H_X^*,
\qquad H_X:= [H_{pq}(v_j)]_{1\le j\le n,(p,q)\in\tau}. $$
The $\sum_{(p,q)\in\tau}\dim(H(p,q))$ columns of $H_X$ are an
(appropriately ordered) orthonormal basis for $P_\tau=\oplus_{(p,q)\in\tau} H(p,q)$ sampled at 
the $n$ points $v_j\in X$. Given the size of $H_X$,
we obtain (\ref{rankMcdn}) from $\rank(M)\le\rank(H_X)$.
 
Now suppose that $X$ is an $A$-code. Then $M=I$, which is rank $n$, and
(\ref{rankMcdn}) implies (\ref{absbndeqn}). 
Further, suppose that there is equality in (\ref{absbndeqn}),
i.e., $H_X$ is square, and hence invertible. 
Since $\diag(f_{pq}I_{\dim(H(p,q)})_{(p,q)\in\tau}$
is congruent to $I$, by Sylvester's law of inertia, its eigenvalues are all
positive, i.e., $f_{pq}>0$, $\forall (p,q)\in\tau$.
Thus $F$ is a potential for $\tau$. 
%By Theorem \ref{Qpqprodtheorem},
By Corollary \ref{Qpqprodthcorollary},
$G=F\overline{Q_{pq}}=FQ_{qp}$ is a potential, with constant
$g_{00}=f_{pq} Q_{pq}(1)$.
%$$ g_{00}
%=\inpro{F\overline{Q_{pq}},1}_\cgeg
%=\inpro{F,Q_{pq}}_\cgeg
%= f_{pq} \inpro{Q_{pq},Q_{pq}}_\cgeg
%= f_{pq} Q_{pq}(1). $$
Thus $F\overline{Q_{pq}}-f_{pq}Q_{pq}(1)$ is a potential, with zero 
constant, and we have
$$ \sum_j\sum_k \bigl(F\overline{Q_{pq}}-f_{pq}Q_{pq}(1)\bigr)(\inpro{v_j,v_k})
= nF(1)\overline{Q_{pq}}(1)-n^2 f_{pq}Q_{pq}(1)
= nQ_{pq}(1)( {1\over n}-f_{pq})
\ge 0, $$
which implies that
\begin{equation}
\label{fpqineq}
f_{pq}\le {1\over n}, \qquad (p,q)\in\tau.
\end{equation}
This gives
$$ 1=F(1)=\sum_{(p,q)\in\tau} f_{pq}Q_{pq}(1)\le {1\over n}\sum_{(p,q)\in\tau}Q_{pq}(1)
={1\over n} \sum_{(p,q)\in\tau}\dim(H(p,q))  =1, $$
so we must have equality in (\ref{fpqineq}) throughout, 
and we obtain (\ref{Fequalform}).
Finally, we observe that $G=F-f_{00}$ is a potential for $\tau$-designs, 
with zero constant, and calculate
$$ 
\sum_j\sum_k G(\inpro{v_j,v_k}) = nF(1)-n^2 f_{00}
%\sum_j\sum_k (F-f_{00}((\inpro{v_j,v_k}) = nF(1)-n^2 f_{00}
=n^2({1\over n}-f_{00})
=
\begin{cases}
0, & (0,0)\in\tau; \cr
n, & (0,0)\not\in\tau,
\end{cases}
$$
so that $X$ is a $\tau$-design if and only if $(0,0)\in\tau$.
\end{proof}

%For convenience, we will 
We will refer to any sequence $(v_j)$ of unit vectors in $\Cd$
as ``a design'', which is warranted since it is indeed a $\{(0,0)\}$-design.
The obvious way to apply Theorem \ref{complexfisherbounds} is for an
$F$ constructed to vanish at certain prescribed angles. 
In this regard, we say that a polynomial $F$ is an {\bf annihilator}
(polynomial) for a design, or a collection of designs, if
$F(1)=1$ and all angles of the designs are roots of $F$.
Sometimes the condition $F(1)=1$ is replaced by $F(1)\ne0$.
Heuristically, we desire that
\begin{itemize}
\item $F$ has a large zero set, i.e., the collection of designs that it annihilates is large.
\item $F\in P_\tau$, for some $P_\tau$ of small dimension (thereby giving a good bound).
\end{itemize}

\begin{corollary} 
\label{anncor}
If $F=\sum_{(p,q)} f_{pq}Q_{pq}$ is an annihilator polynomial of
a design $X$, then
$$ n=|X|\le\dim(P_\tau)=\sum_{(p,q)\in\tau} \dim(H(p,q)), \qquad
\tau:=\{(p,q): f_{pq}\ne0\}, $$
with equality if and only if the angles of $X$ 
are roots of $\sum_{(p,q)\in\tau} Q_{pq}$.
\end{corollary}

This result appears in \cite{RS14} Theorem 4.2, for $\tau=\cS$ a lower set. 
There, a design with an annihilator polynomial 
in $\spam\{z^p\overline{z}:(p,q)\in\cS\}$ (which equals $P_\cS$, 
for $\cS$ a lower set) is called an $\cS$-code.
Bounds which follow from Corollary \ref{anncor} will be called
{\bf absolute bounds}, as they are in the real and projective cases.

We start with an obvious example 
(see \cite{RS14} Corollary 4.3).

\begin{example} 
\label{naiveAcodeexample}
Let $X$ be a design with $m=|A|$ angles $A\subset\{z\in\CC:|z|\le1,z\ne1\}$.
Then
$$ F(z):= \prod_{\ga\in A}{z-\ga\over1-\ga} =\sum_{k=0}^m f_{k0}Q_{k0}(z), $$
gives the estimate 
$$ n=|X|\le\sum_{k=0}^m \dim(H(k,0))
=\sum_{k=0}^m {k+d-1\choose d-1} = {m+d\choose d}. $$
\end{example}

Complex codes $X$ with two and three angles (inner products) have been studied by
\cite{NS16}, \cite{NS18}. For two angles \cite{RS14}, \cite{NS16} gives the bound
\begin{equation}
\label{NS16bound}
n=|X| \le 
\begin{cases}
2d+1, & \hbox{$d$ is odd}; \cr
2d, & \hbox{$d$ is even},
\end{cases}
\end{equation}
which is clearly better than that of Example \ref{naiveAcodeexample}
($m=2$), i.e.,
$$ n=|X| \le {1\over 2}(d+1)(d+2). $$
Interestingly, the only case where these two bounds coincide is for $d=1$,
in which case
$$ F(z)={1\over3} \bigl(Q_{20}^{(1)}(z)+Q_{10}^{(1)}(z)+Q_{00}^{(1)}(z)\bigr)
= {1\over3}(z^2+z+1). $$
The zeros of this polynomial are the primitive third roots of unity,
and the third roots of unity $X=\{1,\go,\go^2\}$ gives the unique
three-vector two angle code for $\CC^1$ attaining the bound (see Table 1, \cite{NS16}).

If $\ga$ is an angle of design, then so is its conjugate $\overline{\ga}$,
i.e., two-angle designs have angles $\{\ga,\overline{\ga}\}$.
These have the same real part, and so are both roots of the polynomial
$$ F(z)=\Re(z)-\Re(\ga)={z+\overline{z}\over2}-{\ga+\overline{\ga}\over2}
={Q_{10}(z)\over 2d}+{Q_{01}(z)\over 2d}-\Re(\ga)Q_{00}(z). $$
Thus, by applying Theorem \ref{complexfisherbounds},
we obtain essentially the bound of (\ref{NS16bound}).

\begin{example} 
\label{twoangleests}
The collection of designs $X$ whose angles have a fixed real part $a\in\RR$,
e.g., two-angle designs with angles $\{\ga,\overline{\ga}\}$, where $\Re(\ga)=a$, 
satisfy the bound
$$ n=|X|\le
\begin{cases}
2d+1, & a\ne 0; \cr
2d, & a=0.
\end{cases}
$$
Further, if there is equality above for $a\ne0$, then
$$ Q_{10}(\ga)+Q_{01}(\ga)+Q_{00}(\ga)
=  d\ga+d\overline{\ga}+1
= 2d\Re(\ga)+1=0 \Implies a=\Re(\ga)=-{1\over 2d}. $$
The two-angle complex designs are an example of complex 
equiangular lines. 
The Gramian matrix which determines a two-angle complex design has a particularly 
simple form, which can be associated naturally with a graph (on its elements)
via a conference matrix.
In this way, the two-angle complex designs 
have been 
classified in \cite{NS16} (also see the corresponding complex equiangular lines in 
\cite{R07}, \cite{W18} Exercises 12.11 and 12.12).
%for the corresponding complex equiangular lines).
\end{example}

If a design $X$ has three angles, then they must be a complex conjugate pair 
$\ga,\overline{\ga}$, with $a=\Re(\ga)=\Re(\overline{\ga})$, 
and a real angle $b\in\RR$. Since $b$ is a root of the polynomial $z-b$, 
we have the following annihilator polynomial
\begin{align} 
F(z) &= (\Re(z)-\Re(\ga))(z-b) \label{threeannipoly} \\
&= {Q_{20}(z)\over d(d+1)} +{Q_{11}(z)\over2d(d+1)}
-\Bigl({b\over2}+a\Bigr){Q_{10}(z)\over d} 
-b {Q_{01}(z)\over 2d} + \Bigl(ab+{1\over 2d}\Bigr)Q_{00}(z). \nonumber
\end{align}
For a generic $a,b$ (all coefficients above nonzero) this gives
\begin{equation}
\label{genericbound}
n \le Q_{20}(1)+Q_{11}(1)+Q_{10}(1)+Q_{01}(1)+Q_{00}(1)
= {d(3d+5)\over 2}.
\end{equation}
The bound given in \cite{NS16} for three angle complex designs is
\begin{equation}
\label{threeanglebound}
n=|X|\le
\begin{cases}
4, & d=1;\cr
d(d+2), & d\ge 2.
\end{cases}
\end{equation}
These agree for $d=1$ ($n\le4$), where the $F$ for four vectors
meeting the bound (\ref{genericbound}) becomes
$$ F(z)= {1\over 4} (z^2+2|z|^2-1+z+\overline{z}), $$
which has roots $-1,i,-i$, which are the three angles of
$X=\{1,-1,i,-i\}\subset\CC^1$.

For $d\ge2$, the bound (\ref{threeanglebound}) improves 
the generic bound (\ref{genericbound}). In particular, for 
$d=2$, the bounds are $n\le8$ and $n\le11$. By specifying 
a particular form for the three angle design, the number of 
terms in (\ref{threeannipoly}) can be reduced, giving a sharper bound.
To get $n\le8$, one must choose
$$ {b\over 2}+a=0, \quad ab+{1\over2d}=0 
 \Implies a=\pm{1\over2\sqrt{d}}, \ 
b=\mp {1\over\sqrt{d}} \ (d=2).  $$
Given that \cite{NS16} show that there is a unique three-angle design
of eight vectors in $\CC^2$ given by the two-angle design of four
vectors with $a=0$ (equiangular lines) multiplied by $\pm1$ ($b=-1$), there
can be no three-angle design with the above parameters.

For $d=3$ the bounds are $n\le15$ and $n\le 21$, whilst the maximal
number of vectors in a three-angle design is calculated to be
$n=9$ (one can take an orthonormal basis multiplied by the three roots of
unity, $a=-1/2$, $b=0$). Moreover, \cite{NS16} Theorem 14 shows that
equality in (\ref{threeanglebound}) can only be obtained for $d=1,2$
(where $b=-1$).

\begin{theorem}
Let $X$ be a %complex 
design for which the real part of its angles can take $s$ possible values. 
This includes the complex designs with $m=2s$ angles, none real.
Then
\begin{equation}
\label{anglerealpartest}
n=|X|\le \sum_{p+q\le s}\dim(H(p,q))
= {s+2d-1\choose 2d-1}+{s+2d-2\choose 2d-1}.
\end{equation}
  %= {s+2d-1\choose s}+{s+2d-2\choose s-1}
\end{theorem}

\begin{proof}
Let $a_1,\ldots,a_s\in\RR$ be the possible real parts of the angles. Then
$$ F(z)
= \prod_{j=1}^s {\Re(z)-a_j\over 1-a_j}
= \prod_{j=1}^s {{z+\overline{z}\over2}-a_j\over 1-a_j}
= \sum_{p+q\le s} f_{pq} Q_{pq}(z), $$
is an annihilator polynomial for $X$. The Gegenbauer expansion above
follows because the ${s+2\choose2}$ polynomials $(Q_{pq}^{(d)}(z))_{p+q\le s}$ 
are linearly independent, and hence are a basis for the space of polynomials of
degree $s$ in the variables $z$ and $\overline{z}$.
Applying Corollary \ref{anncor}, then gives the result, 
where the upper bound is calculated via
\begin{align*}
\sum_{p+q\le s}\dim(H(p,q)) 
&=\dim(\Hom_s(\RR^{2d}))+\dim(\Hom_{s-1}(\RR^{2d})) \\
&={s+2d-1\choose 2d-1}+{s-1+2d-1\choose 2d-1}. 
\end{align*}
\end{proof}

For $s=1$, we recover Example \ref{twoangleests}.

\begin{example}
($s=2$) For four-angle designs (\ref{anglerealpartest}) gives the estimate
$$ n=|X|\le d(2d+3), $$
which has the same growth in $d$ as the estimates for three-angle designs.
As for the case $s=1$, better estimates can be obtained if the values
of the angles are constrained. Indeed
$$ (\Re(z)-a)(\Re(z)-b)= {Q_{20}(z)+Q_{11}(z)+Q_{02}(z)\over 2d(d+1)} 
-{a+b\over 2d}\{Q_{10}(z)+Q_{01}(z)\} 
+\Bigl(ab+{1\over 2d}\Bigr),
$$
so that if $a\ne b$ are the real parts of the angles, then
$$ n \le
\begin{cases}
d(2d+1), & b=-a; \\
d(2d+1)-1, & b=-a=\pm{1\over\sqrt{2d}}.
\end{cases}
$$
\end{example}

We conclude this section, by showing that the estimates 
for the number of lines in $\Cd$ (projective designs)
naturally follow from Corollary \ref{anncor}.
For lines, or projective designs, there is no meaningful notion of
angle $\inpro{v_j,v_k}\in\CC$, but rather the projective invariants
$$ |\inpro{v_j,v_k}|^2=a, $$
are of interest. These are sometimes referred to as angles (particularly
for projective designs), e.g., as in ``complex equiangular lines''.
Fortunately, there simple and natural annihilator polynomials for
this ``angle'', i.e.,
$$ F(z)=|z|^2-a, \quad a\ne 0, \qquad F(z)=z,\ F(z)=\overline{z}, \quad a=0. $$

Let $X=(v_j)$ be unit vectors giving a set of lines in $\Cd$. 
If the angles $|\inpro{v_j,v_k}|^2$, $j\ne k$, take $s$ possible values
$A\subset[0,1)$, then an annihilator polynomial for $X$ is given by
$$ F(z):= z^\gep \prod_{a\in A\setminus\{0\}} {|z|^2-a\over 1-a}
=\sum_{j=0}^{s-\gep} f_{j+\gep,j} Q_{j+\gep,j}(z), 
\qquad \gep=\gep(A):=
\begin{cases}
0, & 0\not\in A; \\
1, & 0\in A.
\end{cases}
$$
Such configurations are said to be an {\bf $A$-set} or {\bf $s$-angular}.
We may apply Corollary \ref{anncor},
to obtain the estimate for $A$-sets given in \cite{DGS77} 
(Theorem 6.1), i.e.,
\begin{equation}
\label{projestimate}
n=|X|\le \sum_{j=0}^{s-\gep}\dim(H(j+\gep,j)).
\end{equation}

\begin{theorem}
Let $X=(v_j)$ be unit vectors giving a set of $s$-angular lines in $\Cd$,
i.e., whose the angles $x=|z|^2=|\inpro{v_j,v_k}|^2$, $j\ne k$, 
take $s$ possible values $A\subset[0,1)$. Then
\begin{equation}
\label{s-angularbound}
n=|X|\le {s+d-1\choose d-1}{s-\gep+d-1\choose d-1}, \qquad
\gep=\gep(A):=
\begin{cases}
0, & 0\not\in A; \\
1, & 0\in A.
\end{cases}
\end{equation}
Further, the angles of the sets of $s$-angular lines giving equality in
(\ref{s-angularbound}) are roots of
%$x=|z|^2=|\inpro{v_j,v_k}|$ which are the roots of 
\begin{equation}
\label{tights-angularpoly}
f(x)= x^\gep P_{s-\gep}^{(d-1,\gep)}(2x-1).
\end{equation}
\end{theorem}

\begin{proof}
%The right hand side of 
The upper bound of (\ref{projestimate}) simplifies, %can be simplified, 
using Lemma \ref{complexGegsumslemma}
and (\ref{Hpqdimension}), to
\begin{align*}
\sum_{j=0}^{s-\gep} Q_{j+\gep,j}^{(d)}(1)
&= {d\over 2s-\gep+d} Q_{s,s-\gep}^{(d+1)}(1) \\
&= {d\over 2s-\gep+d} {(2s-\gep+d)(s+d-1)!(s-\gep+d-1)! \over s!(s-\gep)!d!(d-1)!} \\
& = {s+d-1\choose d-1}{s-\gep+d-1\choose d-1},
\end{align*}
which gives (\ref{s-angularbound}). The
(tight) $s$-angular designs giving 
equality in (\ref{s-angularbound}) have inner products
$z=\inpro{v_j,v_k}$, $j\ne k$, which are roots of the polynomial
\begin{equation}
\label{s-angulartight}
\sum_{j=0}^{s-\gep} Q_{j+\gep,j}^{(d)}(z) 
= {d\over 2s-\gep+d} Q_{s,s-\gep}^{(d+1)}(z). 
\end{equation}
Therefore, in view of (\ref{Homkekequation}),
the $s$ distinct possible angles $x=|z|^2=|\inpro{v_j,v_k}|$ 
are roots of the polynomial (\ref{tights-angularpoly}).
\end{proof}

By way of comparison with (\ref{s-angularbound}), 
the bound (\ref{new(t,t)-designlowerbound}) 
for spherical $(t,t)$-designs can be written as 
$$ n=|X|\ge {s+d-1\choose d-1}{s-\gep+d-1\choose d-1}, \qquad
t=2s-\gep, \quad \gep=0,1, $$
where the angles $|z|=|\inpro{v_j,v_k}|$ of the tight spherical $(t,t)$-designs 
are roots of the same polynomial (\ref{s-angulartight}).
In this way, 
\begin{itemize}
\item There is a $1$--$1$ correspondence between the tight complex spherical 
$(t,t)$-designs with $t=2s-\gep$, $\gep=0,1$, and the tight $s$-angular 
designs which have $0$ as an angle if and only if $\gep=1$.
\end{itemize}

Upper bounds obtained from annihilator polynomials which are potentials
are known as {\bf special bounds}. 
We now show, by example, how the special and absolute bounds for $s$-angular lines
(projective designs) can be obtained directly from our general results.

The annihilator polynomial for one angle
$$ F(z)=(|z|^2-a)
= {Q_{11}(z)\over d(d+1)} + {1-ad\over d} Q_{00}(z), $$
gives the absolute bounds
$$ n\le Q_{11}(1)+Q_{00}(1)=d^2, \quad a\ne{1\over d}, \qquad
n \le Q_{21}(1) = d^2-1, \quad a={1\over d}. $$
We observe that for $d=2$ there are tight such designs, four complex 
equiangular lines and the vertices of the equilateral triangle, respectively.
We may apply Theorem \ref{complexupperlowerbnds},
subject to the condition that the above $F$ is a potential, 
to obtain the special bounds
$$ n\le {F(1)\over f_0}= {1-a\over {1-ad\over d} }
= {d(1-a)\over 1-da}, \qquad a > {1\over d}. $$

The special bound holds for $a=0$, as $n\le d$, which is
considerably sharper than the absolute bound, but this cannot be proved 
by taking the limit $a\to0$. It can be proved using 
Example \ref{conjzFresult} applied to annihilator polynomial $F(z)=z$.
We illustrate this general process for $2$-angular lines.
Consider the following annihilator polynomials lines with angles
$\{0,a\}$ and $\{a,b\}$
\begin{align*}
F(z) & = z(|z|^2-a) \cr
&= {Q_{21}(z)\over {1\over2}d(d+1)(d+2)}+{Q_{10}(z)\over d}
\Bigl({2\over d+1}-a\Bigr), \cr
F(z) 
&= (|z|^2-a)(|z|^2-b) \cr
& = { Q_{22}(z)\over{1\over4} d(d+1)(d+2)(d+3)}
+\Bigl(-a-b+{4\over d+2}\Bigr) {Q_{11}(z)\over d(d+1)} \cr
& \qquad
+ {d(d+1)ab-(d+1)(a+b)+2\over d(d+1)} Q_{00}(z).
\end{align*}
By using Example \ref{conjzFresult}, and applying 
Theorem \ref{complexupperlowerbnds} to the potential $\overline{z}F(z)$,
the first gives the special bound
\begin{equation}
\label{specialboundI}
n\le {F(1)\over f_{10}} = {1-a\over {1\over d}({2\over d+1}-a)}
={d(d+1)(1-a)\over 2-(d+1)a}, \qquad
a< {2\over d+1},
\end{equation}
for sets with angles $\{0,a\}$. For sets with angles $\{a,b\}$,
direct application of Theorem \ref{complexupperlowerbnds} to the
second potential $F$ gives the special bounds
\begin{equation}
\label{specialboundII}
n\le {F(1)\over f_{00}} 
= {(1-\ga)(1-\gb)\over {d(d+1)ab-(d+1)(a+b)+2\over d(d+1)} }
= {d(d+1)(1-\ga)(1-\gb) \over d(d+1)ab-(d+1)(a+b)+2},
\end{equation}
which holds for
$$ a+b\le {4\over d+2}, \qquad d(d+1)ab-(d+1)(a+b)+2>0. $$
We observe that the limit $\gb\to0$ of this bound, gives the bound 
for angles $\{0,\ga\}$, but does not prove it.
The bounds (\ref{specialboundI}) and (\ref{specialboundII}) 
were originally obtained using the polynomials $Q_k^\gep$ of 
(\ref{Qkgepdefn}) (see \cite{DGS91}, \cite{H92}).

\bibliographystyle{alpha}
\bibliography{references}
\nocite{*}
\vfil
\end{document}